\documentclass[a4paper]{easychair}

\usepackage{amsfonts, amsmath, amsthm, amssymb} 
\usepackage{hyperref}
\usepackage{enumitem}
\setlist[enumerate]{label=(\roman*), leftmargin=*}
\usepackage{algorithm}
\usepackage{algpseudocode}
\usepackage{float} 
\algrenewcommand\algorithmicrequire{\textbf{Input:}}
\algrenewcommand\algorithmicensure{\textbf{Output:}}
\usepackage{tikz-cd}
\usepackage{tikz}
\definecolor{linkcolor}{RGB}{170,0,0}
\definecolor{citecolor}{RGB}{44,160,46}
\definecolor{urlcolor}{RGB}{0,51,153}
\hypersetup{
  linkcolor=linkcolor,
  citecolor=citecolor,
  urlcolor=urlcolor
}
\usetikzlibrary{patterns,arrows,fit}
\usetikzlibrary{arrows.meta}
\usetikzlibrary{intersections,calc}
\usetikzlibrary{automata, positioning}
\usepackage{etoolbox}

\newtheorem{theorem}{Theorem}[section]
\newtheorem{lemma}[theorem]{Lemma}
\newtheorem{proposition}[theorem]{Proposition}
\newtheorem{corollary}[theorem]{Corollary}

\newtheorem{remark}[theorem]{Remark}
\theoremstyle{plain}
\newtheorem{definition}[theorem]{Definition}
\newtheorem{example}[theorem]{Example}
\newtheorem{claim}[theorem]{Claim}

\newcommand{\ve}{\varnothing}
\newcommand{\powerset}[1]{{\mathcal{P}(#1)}}

\newcommand{\rsto}{{\upharpoonright}}
\newcommand{\sub}{\subseteq}
\newcommand{\set}[1]{{\{ #1 \}}}
\newcommand{\tup}[1]{{\langle #1 \rangle}}

\newcommand{\inver}[1]{{{#1}^{-1}}}

\newcommand{\Zp}{{\mathbb{Z}^+}}

\newcommand{\ModelStyle}[1]{\mathfrak{#1}}
\newcommand{\mm}{\ModelStyle{M}}

\newcommand{\FrameStyle}[1]{{\mathfrak{#1}}}
\newcommand{\ff}{{\FrameStyle{F}}}

\newcommand{\Ch}{\FrameStyle{Ch}}

\newcommand{\Cl}{\FrameStyle{Cl}}

\newcommand{\GeneralFrameStyle}[1]{\mathbb{#1}}
\newcommand{\gf}{\GeneralFrameStyle{F}}
\AtBeginDocument{
  
}

\newcommand{\ClassOfStructureStyle}[1]{\mathcal{#1}}
\newcommand{\CK}{\ClassOfStructureStyle{K}}

\newcommand{\FunctorStyle}[1]{\mathsf{#1}}

\newcommand{\Fin}{\FunctorStyle{Fin}}
\newcommand{\Fr}{\FunctorStyle{Fr}}

\newcommand{\RFr}{\FunctorStyle{RFr}}
\newcommand{\GFr}{\FunctorStyle{GFr}}

\newcommand{\Log}{\FunctorStyle{Log}}

\newcommand{\NExt}{\mathop{\FunctorStyle{NExt}}}
\newcommand{\Ext}{\mathop{\FunctorStyle{Ext}}}

\newcommand{\E}{\mathsf{E}}
\newcommand{\A}{\mathsf{A}}

\AtBeginDocument{

}

\newcommand{\Prop}{\mathsf{Prop}}

\AtBeginDocument{
\renewcommand{\phi}{\varphi}
}

\newcommand{\bd}{\blacklozenge}
\newcommand{\bb}{\blacksquare}

\newcommand{\D}{\Diamond}
\newcommand{\B}{\Box}

\newcommand{\is}[1]{{\Delta^{\leq #1}}}
\newcommand{\all}[1]{{\nabla^{\leq #1}}}

\newcommand{\SucPre}[1]{{#1}_\sharp}
\newcommand{\R}{\SucPre{R}}

\newcommand{\axiom}[1]{\mathsf{#1}}

\newcommand{\IL}[1]{\mathsf{#1}}
\newcommand{\ML}[1]{\mathsf{#1}}
\newcommand{\TL}[1]{\mathsf{#1}_{t}}

\newcommand{\FormT}{\mathsf{Form}_{t}}

\newcommand{\md}{\vDash}
\newcommand{\nmd}{\nvDash}

\newcommand{\M}{\mathcal{M}}
\newcommand{\Reach}{\mathcal{R}}

\usepackage{mathrsfs}
\newcommand{\Lt}{\mathscr{L}_t}

\newcommand{\commentout}[1]{}

\usepackage{xcolor}

\definecolor{newresult}{RGB}{0,94,176}
\definecolor{previousresult}{RGB}{212,92,0}

\newcommand{\newresult}[1]{\textcolor{newresult}{#1}}
\newcommand{\previousresult}[1]{\textcolor{previousresult}{#1}}

\newtheorem*{maintheorem}{Theorem \ref*{thm:undec-scheme-grzt}}
\newtheorem*{maintheoremtab}{Theorem \ref*{thm:undec-prop-S4t-=tab}}

\newcommand{\AB}{\set{a_i,b_i : i \in \omega}}
\newcommand{\tri}{\mathrm{tri}}

\setlist[enumerate]{label=(\arabic*)}

\title{
Most properties are undecidable even in $\NExt\TL{Grz}$
}

\author{
    Qian Chen\inst{1} and Tenyo Takahashi\inst{2}
}

\institute{
Department of Philosophy, Xiamen University \\
\href{mailto:chenq9901@xmu.edu.cn}{chenq9901@xmu.edu.cn}
\and
Institute for Logic, Language and Computation, University of Amsterdam \\
\href{mailto:t.takahashi@uva.nl}{t.takahashi@uva.nl}
}

\authorrunning{~}

\titlerunning{~}

\fancypagestyle{plain}{%
    \fancyhf{}%
    \fancyfoot[C]{\normalsize\thepage}%
}

\begin{document}

\maketitle

\begin{abstract}
    We investigate decidability of properties in the lattice $\NExt\TL{Grz}$ of extensions of the Grzegorczyk tense logic $\TL{Grz}$ and the lattice $\NExt\TL{S4}$ of reflexive and transitive tense logics, with applications to the lattice $\Ext\IL{biIPC}$ of bi-superintuitionistic logics. We prove that a broad class of properties is undecidable in $\NExt\TL{Grz}$, including tabularity, Kripke completeness, the finite model property, and decidability, which also yields their undecidability in $\NExt\TL{S4}$. We also construct infinitely many tabular extensions of $\TL{Grz}$ (and thus of $\TL{S4}$) whose coincidence problems are undecidable, while presenting one tabular extension of $\TL{Grz}$ and infinitely many ones of $\TL{S4}$ with a decidable coincidence problem. As a consequence, we obtain that the finite model property and tabularity are undecidable in $\Ext\IL{biIPC}$, and that there are infinitely many tabular extensions of $\IL{biIPC}$ whose coincidence problems are undecidable. These results clarify some similarities and differences between $\NExt\TL{Grz}$ and $\NExt{\ML{Grz}}$, $\NExt\TL{S4}$ and $\NExt\ML{S4}$, as well as $\Ext\IL{biIPC}$ and $\Ext\IL{IPC}$.  
    The proofs adapt Chagrov's method of reducing from an undecidable problem for Minsky machines. We isolate and explicitly formulate the method of good valuations, a recurring technique underlying several proofs in the literature that use large frames, making it available for further applications.
\end{abstract}

\section{Introduction}

This paper studies (un)decidability of properties for tense extensions of $\TL{Grz}$.
Decidability of properties of logics has been extensively studied for modal logics; see, e.g., \cite[Section 17.6]{Chagrov.Zakharyaschev1997} and \cite{Wolter.Zakharyaschev2007} for historical notes and surveys. Recall that a property $P$ in a set $\mathcal{C}$ of logics is identified with the subset $\{L \in \mathcal{C}: \text{$L$ has $P$}\}$. The property $P$ is \emph{decidable} in $\mathcal{C}$ if there is an algorithm which, for every finitely axiomatizable logic $L \in \mathcal{C}$ presented by a finite axiomatization, decides whether $L$ has the property $P$. For a modal logic $L$, let $\NExt L$ denote the lattice of all normal extensions of~$L$.

In the unimodal setting, most properties are undecidable in $\NExt{\mathsf{K}}$, including Kripke completeness, the finite model property (FMP), decidability, tabularity, and coincidence with a fixed tabular logic (see, e.g., \cite[Chapter 17]{Chagrov.Zakharyaschev1997}). On the positive side, decidable properties in $\NExt{\ML{K}}$ include consistency \cite{Makinson1971}, being a union-splitting, and strict Kripke completeness \cite{Takahashi2026}. 

Restricting to a smaller lattice of extensions can turn undecidable properties decidable. For instance, in the sublattice $\NExt{\mathsf{K4}}$ of $\NExt{\ML{K}}$, coincidence with a fixed tabular logic is decidable, as every tabular logic has only finitely many immediate predecessors and all of them are tabular \cite{Blok1980} (see also \cite[Theorem 17.3]{Chagrov.Zakharyaschev1997}). As far as we know, the decidability of tabularity in $\NExt{\ML{K4}}$ remains open \cite{rautenbergWillemBlokModal2006}, while it becomes decidable in the further sublattice $\NExt{\ML{S4}}$ \cite{Maksimova1975a}.
On the other hand, Chagrov and Zakharyaschev~\cite{Chagrov.Zakharyaschev1993} showed that the FMP and decidability are undecidable even in the sublattice $\NExt\ML{Grz}$ of $\NExt{\ML{S4}}$, where the Grzegorczyk logic $\ML{Grz}$ is the greatest modal companion of the intuitionistic logic $\IL{IPC}$. Note that the undecidability of Kripke completeness does not follow from their proofs directly and remains open.

Tense logics are normal bimodal logics extending the logic $\TL{K} = \mathsf{K}_2 \oplus \set{p \to \B\bd p, p \to \bb\D p}$, where $\mathsf{K}_2$ is the least normal bimodal logic with two modalities $\B$ and $\bb$. The intended meanings of $\B$ and $\bb$ are ``always true in the future'' and ``always true in the past'', respectively. The modern modal approach to logic of time was initiated by Prior~\cite{Prior1967,Prior1968}. Among polymodal logics, tense logics form one of the most well-studied classes; see, for example, \cite{Goldblatt1992,Kracht1999l,Burgess2002,Wolter1993,Chen2026a}. In this paper, we are mostly interested in $\TL{Grz}$, where in general $L_t$ is the minimal tense logic whose $\B$- and $\bb$-fragments both contain $L$.

Despite the minimal tense extension map $(\cdot)^+: \NExt{\mathsf{K}} \to \NExt{\TL{K}}$; $\mathsf{K} \oplus \phi \mapsto \TL{K} \oplus \phi$, and an embedding $\mathsf{sim}: \NExt\ML{K}_2 \to \NExt\ML{K}$ studied by Kracht and Wolter~\cite{Kracht.Wolter1999} (see also \cite{Thomason1974a}), the relationship between the (un)decidability of properties in $\NExt{L}$ and in $\NExt{L_t}$ is by no means obvious. The map $(\cdot)^+$ is not injective~\cite{Wolter1993} and preserves neither Kripke completeness nor the FMP \cite{Wolter1996}. The image $\mathsf{sim}[\NExt\TL{K4}]$ is not contained in $\NExt\ML{K4}$. Thus, we may not transfer the (un)decidability results directly from the unimodal setting to the tense setting. In general, the interactions between tense modalities make lattices of tense logics completely different from those of unimodal logics (see, e.g., \cite{Kracht1992,Ma.Chen2021,Ma.Chen2023,Chen.Ma2024}).

Chagrov and Shehtman~\cite{Chagrov.Shehtman1995} proved that tabularity, coincidence with a fixed tabular tense logic, and consistency are undecidable in $\NExt{\TL{K4}}$. Recently, the authors \cite{Chen.Takahashi2026} extended their method to obtain a general criterion for a property to be undecidable in $\NExt{\TL{K4}}$, which yields the undecidability of many properties, including the aforementioned ones as well as Kripke completeness, the FMP, and decidability. These results contrast the unimodal setting and the tense setting. Given these undecidability results in $\NExt{\TL{K4}}$ and the previous results in $\NExt{\ML{S4}}$ and $\NExt{\ML{Grz}}$, it is natural to ask about the decidability of these properties in $\NExt{\TL{S4}}$ and $\NExt{\TL{Grz}}$. 

An extra motivation to focus on ${\TL{Grz}}$ comes from its relation to the \textit{bi-intuitionistic logic} ${\IL{biIPC}}$ \cite{Rauszer1974}. Similar to the well-known Blok-Esakia theorem, which states that $\NExt{\ML{Grz}}$ is isomorphic to the lattice $\Ext{\IL{IPC}}$ of superintuitionistic logics, there is also a lattice isomorphism between $\NExt{\TL{Grz}}$ and $\Ext{\IL{biIPC}}$ \cite{wolterLogicsCoimplication1998,BezhanishviliBLOKESAKIATHEOREMSSTABLE2025a}. In particular, $\TL{Grz}$ is the greatest tense companion of $\IL{biIPC}$.

In this paper, we prove a general criterion for a property to be undecidable in $\NExt\TL{Grz}$, which implies the undecidability of many properties, including tabularity, Kripke completeness, the FMP, and decidability. It follows that these properties are also undecidable in $\NExt{\TL{S4}}$. Moreover, we prove that there exist infinitely many tabular logics in $\NExt{\TL{Grz}}$ with an undecidable coincidence problem. On the other hand, we show that there is a tabular logic in $\NExt\TL{Grz}$ whose coincidence problem is decidable, and there are infinitely many such logics in $\NExt{\TL{S4}}$. These results clarify some similarities and differences between the lattices $\NExt\TL{Grz}$ ($\NExt{\TL{S4}}$) and $\NExt{\ML{Grz}}$ ($\NExt{\ML{S4}}$). Moreover, transferring the results by the Blok-Esakia isomorphism between $\NExt{\TL{Grz}}$ and $\Ext\IL{biIPC}$ \cite{wolterLogicsCoimplication1998,BezhanishviliBLOKESAKIATHEOREMSSTABLE2025a}, we obtain the undecidability of the FMP and tabularity in $\Ext{\IL{biIPC}}$. It also follows that there are infinitely many tabular logics in $\Ext\IL{biIPC}$ with an undecidable coincidence problem. These results, again, show some similarities and differences between the lattices $\Ext\IL{biIPC}$ and $\Ext\IL{IPC}$.

Our proof adapts Chagrov's method of reducing an undecidable problem of Minsky machines, which was used to prove other undecidability results discussed in this introduction as well. However, there is a challenge regarding definability on frames compared to the case of $\NExt{\TL{K4}}$ \cite{Chen.Takahashi2026}. In that case, to simulate the behavior of a Minsky machine on a general frame, we defined points by variable-free formulas, whose truth is uniform for valuations. However, this is not possible for $\TL{S4}$-frames since all points are reflexive. To resolve this issue, we will employ a recurring technique underlying many proofs using large frames, such as other undecidability proofs based on Chagrov's method \cite{Chagrov.Zakharyaschev1993}, Fine's incompleteness proof \cite{Fine1974IncompleteLogicContainingS4}, and the proof of Blok's dichotomy theorem in \cite[Theorem 10.59]{Chagrov.Zakharyaschev1997}. The idea is to insert a special formula $\phi^*$ when constructing the reduction. It allows us to consider, in practice, only ``good'' valuations, i.e., valuations $V$ with $V(\phi^*) \neq \ve$. Then we can define points (using variables) uniformly for these good valuations. We formulate this method of \emph{good valuations} explicitly so that it is convenient for further applications. In the proofs in \cite{Chagrov.Zakharyaschev1993} for $\NExt{\ML{Grz}}$, good valuations are realized using a \emph{subframe formula} (see, e.g., \cite[Section 9.4]{Chagrov.Zakharyaschev1997}), which is not available for tense logics. Instead, we will carefully design our general frame that simulates the Minsky machine and exploit the expressiveness of tense logics to define $\phi^*$.

Table~\ref{table:UndecidableProperties} summarizes the results on the (un)decidability of some major properties discussed so far in lattices of modal and tense logics. The results established in this paper are highlighted in \textcolor{newresult}{blue}, while those established in our previous paper \cite{Chen.Takahashi2026} are highlighted in \textcolor{previousresult}{orange}.

\begin{table}[htbp]
    \centering
    {\renewcommand{\arraystretch}{1.5}
    \begin{tabular}{c|c|c|c|c|c|c}
        & Cons. & Tab. & Fixed Tab. & KC & FMP & Dec. \\ \hline
        $\NExt{\ML{K}}$ 
        & $\checkmark$ & $\times$ & $\times$ & $\times$ & $\times$ & $\times$ \\ \hline

        $\NExt{\ML{K4}}$ 
        & $\checkmark$ & ? & $\checkmark$ & $\times$ & $\times$ & $\times$ \\ \hline

        $\NExt{\ML{S4}}$ 
        & $\checkmark$ & $\checkmark$ & $\checkmark$ & ? & $\times$ & $\times$ \\ \hline

        $\NExt{\ML{Grz}}$ 
        & $\checkmark$ & $\checkmark$ & $\checkmark$ & ? & $\times$ & $\times$ \\ \hline


        $\NExt{\TL{K}}$
        & $\times$
        & $\times$
        & $\times$
        & \previousresult{$\times$}
        & \previousresult{$\times$}
        & \previousresult{$\times$} \\ \hline
        
        $\NExt{\TL{K4}}$
        & $\times$
        & $\times$
        & $\times$
        & \previousresult{$\times$}
        & \previousresult{$\times$}
        & \previousresult{$\times$} \\ \hline



        $\NExt{\TL{S4}}$
        & $\checkmark$
        & \newresult{$\times$}
        & \newresult{depends}
        & \newresult{$\times$}
        & \newresult{$\times$}
        & \newresult{$\times$} \\
        \hline
        
        $\NExt\TL{Grz}$
        & $\checkmark$
        & \newresult{$\times$}
        & \newresult{depends}
        & \newresult{$\times$}
        & \newresult{$\times$}
        & \newresult{$\times$} \\





    \end{tabular} 
    }

    \vspace{0.5em}
    {\footnotesize
    \textbf{Abbreviations.}
    Cons.: consistency;
    Tab.: tabularity;
    Fixed Tab.: coincidence with a fixed tabular logic;\\
    KC: Kripke completeness;
    FMP: finite model property;
    Dec.: decidability.
    }
    \caption{(Un)decidability of properties for modal and tense logics}
    \label{table:UndecidableProperties}
\end{table}

We also provide in Table~\ref{table:UndecidableProperties-IPC} an analogous overview in the lattices of superintuitionistic and bi-superintuitionistic logics \cite{Maksimova1972,Chagrov.Zakharyaschev1993}. Results proved in this paper are highlighted in \textcolor{newresult}{blue}.

\begin{table}[htbp]
    \centering
    {\renewcommand{\arraystretch}{1.5}
    \begin{tabular}{c|c|c|c|c|c|c}
        & Cons. & Tab. & Fixed Tab. & KC & FMP & Dec. \\ \hline

        $\Ext{\IL{IPC}}$ 
        & $\checkmark$ & $\checkmark$ & $\checkmark$ & ? & $\times$ & $\times$ \\ \hline

        $\Ext{\IL{biIPC}}$
        & $\checkmark$
        & \newresult{$\times$}
        & \newresult{depends}
        & ?
        & \newresult{$\times$}
        & ? \\


    \end{tabular} 
    }
    
    \vspace{0.5em}
    {\footnotesize
    \textbf{Abbreviations.}
    Cons.: consistency;
    Tab.: tabularity;
    Fixed Tab.: coincidence with a fixed tabular logic;\\
    KC: Kripke completeness;
    FMP: finite model property;
    Dec.: decidability.
    }
    \caption{(Un)decidability of properties for superintuitionistic and bi-superintuitionistic logics}
    \label{table:UndecidableProperties-IPC}
\end{table}

This paper is organized as follows. Section~\ref{sec 2} summarizes necessary preliminaries. In Section~\ref{sec 3}, we briefly discuss some decidable properties in $\NExt{\TL{Grz}}$ and $\NExt{\TL{S4}}$. Section~\ref{sec:S4t--chap:Undecidability} focuses on undecidable properties in $\NExt{\TL{Grz}}$ and is the main part of the paper. In Section~\ref{subsec:main theorems}, we discuss the two main theorems, Theorems~\ref{thm:undec-scheme-grzt} and \ref{thm:undec-prop-S4t-=tab}, and their consequences. The rest of Section~\ref{sec:S4t--chap:Undecidability} is devoted to proving these theorems. 

This paper is partially based on \cite[Chapter~6.3]{Chen2026a}.

\section{Preliminaries} \label{sec 2}

\subsection{Tense logics} \label{sec 2.1}

We recall the basics of tense logic in this subsection. We refer to \cite{Chagrov.Zakharyaschev1997,Kracht1999l,Blackburn.deRijke.ea2001} for modal logic in general. The formal language of tense logic is obtained by adding two modal operators $\B$ and $\bb$ to the language of propositional logic. Thus, a unimodal formula can be viewed as a tense formula in the natural way. 
A \emph{tense logic} is a normal bimodal logic that contains the axioms $p \to \B\bd p$ and $p \to \bb\D p$. Let $\TL{K}$ denote the \emph{minimal tense logic}. For each tense logic $L$, let $\NExt L$ denote the lattice of all normal extensions of $L$. For every tense logic $L$ and set of formulas $\Sigma$, we write $L\oplus\Sigma$ for the smallest tense logic containing $L\cup\Sigma$. We write $L \oplus \phi$ for $L \oplus \set{\phi}$. A tense logic $L$ is \emph{finitely axiomatizable} if $L = \TL{K} \oplus \phi$ for some $\phi \in \Lt$.
A tense logic $L$ is \emph{consistent} if $\bot\not\in L$, and so the only inconsistent tense logic is $\Lt$.

A \emph{general frame} is a triple $\gf=(X,R,A)$ where $X$ is a non-empty set, $R$ a binary relation on $X$ and $A$ a subset of $\powerset{X}$ such that (i) $\ve \in A$, and (ii) $A$ is closed under Boolean operations on $\powerset{X}$, $R[\cdot]$ and $\inver{R}[\cdot]$, where $R[Y] \coloneqq \set{x \in X: \exists{y\in Y}(Ryx)}$ and $\inver{R}[Y] \coloneqq \set{x \in X: \exists{y\in Y}(Rxy)}$ for all $Y \sub X$. A general frame $\gf = (X,R,A)$ is called a \emph{refined frame} if it is (i) \emph{differentiated}: for all distinct points $x,y \in X$, there exists $U \in A$ such that $x \in U$ and $y \notin U$; and (ii) \emph{tight}: for all points $x,y \in X$, if $(x,y) \notin R$, then there exists $U \in A$ such that $y \in U$ and $x \notin \inver{R}[U]$. A \emph{Kripke frame} $\ff$ is a general frame of the form $(X,R,\powerset{X})$ and we simply write $(X,R)$. Let $\GFr$, $\RFr$, $\Fr$, and $\Fin$ denote the classes of all general frames, refined frames, Kripke frames, and finite Kripke frames, respectively.

A \emph{model} is a pair $\mm=(\gf,V)$ where $\gf\in\GFr$ and $V: \Prop \to A$ a valuation on $\gf$. $V$ is extended to $V:\Lt\to A$ as usual: $V(\bd\phi) = R[V(\phi)]$ and $V(\B\phi) \coloneqq X \setminus \inver{R}[X\setminus V(\phi)]$. The expressions $\mm,x \md \phi$, $\gf,x \md \phi$, $\gf \md \phi$ and $\gf \md \Sigma$ are defined as usual. Note that $\mm,x \md \bd \phi$ if and only if $\mm,y \md \phi$ for some $y \in \inver{R}[x]$. For all sets $\Sigma\sub\Lt$ of formulas and classes $\mathcal{K}\sub\GFr$ of general frames, let 
    \begin{center}
        $\mathcal{K}(\Sigma) \coloneqq \set{\gf\in\mathcal{K}:\gf\vDash\Sigma}$ and $\mathsf{Log}(\mathcal{K}) \coloneqq \set{\phi:\mathcal{K}\vDash\phi}$.
    \end{center}
For example, given a tense logic $L$, we write $\Fin(L)$ for the class of all finite frames for $L$. We call $\mathsf{Log}(\mathcal{K})$ the \emph{tense logic of $\mathcal{K}$}.

For a tense formula $\phi$, let $\phi^\partial$ denote the tense dual of $\phi$ obtained from $\phi$ by interchanging $\B$ and $\bb$. In particular, for a unimodal formula $\phi$, the formula $\phi^\partial$ is obtained from $\phi$ by replacing $\B$ with $\bb$. Recall that for each frame $\ff = (X,R)$, we write $\inver{\ff}$ for its inverse frame $(X,\inver{R})$. The reader can readily check that for all $\phi \in \FormT$, we have $\ff\md\phi$ if and only if $\inver{\ff}\md\phi^\partial$.

For a unimodal logic $L = \ML{K} \oplus \phi$ (where $\ML{K}$ is the least normal unimodal logic), let $L^+$ and $L_t$ denote the tense logics $\TL{K}\oplus\phi$ and $\TL{K} \oplus \phi \oplus \phi^\partial$, respectively. One may think of $L^+$ as the minimal tense extension of $L$, and $L_t$ as the tense counterpart of $L$. Note that for modal logics such as $\ML{K}$, $\ML{K4}$, and $\ML{S4}$, their minimal extensions coincide with their tense counterpart, while in general this is not true.

Recall the \emph{Grzegorczyk formula} $\axiom{grz}$ defined as follows:
\begin{align*}
\axiom{grz} &\coloneqq \B(\B(p\to\B p)\to p)\to p.
\end{align*}
It is known that a frame $\ff$ validates $\axiom{grz}$ if and only if $\ff$ is reflexive, transitive, antisymmetric, and \emph{Noetherian}, meaning that it contains no infinite ascending chain (see, e.g., \cite[Proposition~3.48]{Chagrov.Zakharyaschev1997}). It follows that the following proposition holds.

\begin{proposition}\label{prop:grz}
    Let $\ff=(X,R)$ be a reflexive and transitive frame. 
    \begin{enumerate}
        \item $\ff\md\axiom{grz}$ if and only if $\ff$ is antisymmetric and Noetherian;
        \item $\ff\md\axiom{grz}^\partial$ if and only if $\inver{\ff}$ is antisymmetric and Noetherian;
        \item for all $x\in X$, we have $\ff,x\md\axiom{grz}^\partial$ if and only if $\inver{\ff},x\md\axiom{grz}$.
    \end{enumerate}
\end{proposition}

In this paper, we will be mainly working with extensions of the logic $\TL{Grz} = \TL{K} \oplus \axiom{grz} \oplus \axiom{grz}^\partial$. Note that it is an extension of $\TL{S4}$.

Below, we list some properties for tense logics studied in the literature. Let $L$ be a tense logic. Then (i) $L$ is \emph{Kripke complete}, if $L = \Log(\Fr(L))$; (ii) $L$ has the \emph{finite model property} (\emph{FMP}), if $L = \Log(\Fin(L))$; (iii) $L$ is \emph{tabular}, if $L = \Log(\ff)$ for some finite frame $\ff$; (iv) $L$ is \emph{canonical}, if $L=\mathsf{Log}(\ff^L)$, where $\ff^L$ is the \emph{canonical frame} for $L$; (v) $L$ is \emph{elementary}, if $L = \Log(\CK)$, where $\CK$ is a class of Kripke frames defined by a set of first-order sentences. Moreover, we say that $L$ is \emph{locally tabular} if for each $n \in \omega$, $L$ contains only finitely many non-$L$-equivalent formulas built up from the propositional variables $p_{0}, \cdots, p_{n-1}$. We say that $L$ is \emph{decidable} if there is an algorithm that, given a formula $\phi$, decides whether $\phi \in L$. This \emph{membership decidability} is a property for logics, and should not be confused with decidability of logics' properties; in particular, it is legitimate to say decidability, as a property, is decidable or not. Finally, for a logic $L$, \emph{coincidence with $L$} refers to the property $\set{L}$.


Finally, let us recall some useful formulas that bound the complexity of frames. For any general frame $\gf = (X,R,A)$, a sequence $\mathcal{Y}=\tup{y_i\in X:i<\alpha}$ is called a (i) \emph{strict chain} in $\gf$ if $Ry_\lambda y_\gamma$ for all $\lambda<\gamma<\alpha$ and $y_\lambda\not\in R[y_\gamma]$ for all $\lambda<\gamma<\alpha$; (ii) an \emph{anti-chain} in $\gf$ if $y_\lambda\not\in R[y_\gamma]$ for all $\lambda\neq\gamma<\alpha$.
For each $n > 0$, we define formulas $\axiom{bw}_n$ by
    \begin{align*}
        \axiom{bw}_n &\coloneqq \bigwedge_{i\leq n}\D p_i\to\bigvee_{i\neq j\leq n}\D(p_i\wedge(p_j\vee\D p_j)).
    \end{align*}
Moreover, we recursively define formulas $\axiom{bd}_n$ for $n > 0$ as follows:
    \begin{align*}
        \axiom{bd}_1 &\coloneqq \D\B q_0\to q_0,\\
        \axiom{bd}_{k+1} &\coloneqq \D(\B q_{k}\wedge\neg\axiom{bd}_k)\to q_{k}.
    \end{align*}
These formulas, together with the tense dual $\axiom{bw}^\partial_n = \bigwedge_{i\leq n}\bd p_i\to\bigvee_{i\neq j\leq n}\bd(p_i\wedge(p_j\vee\bd p_j))$ of $\axiom{bw}_n$, bound the complexity of frames in the following sense. For any rooted refined frame $\gf=(X,R,A)$ validating $\TL{S4}$ and $n > 0$, the following hold (see, e.g., \cite[Chapter 3]{Chagrov.Zakharyaschev1997}).
    \begin{itemize}
        \item $\gf\vDash\axiom{bd}_n$ iff $\gf$ contains no strict chain of size greater than $n$;
        \item $\gf\vDash\axiom{bw}_n$ iff for any $x \in X$, $R[x]$ contains no strict anti-chain of size greater than $n$;
        \item $\gf\vDash\axiom{bw}^\partial_n$ iff for any $x \in X$, $R^{-1}[x]$ contains no strict anti-chain of size greater than $n$.
    \end{itemize}
For each general frame $\gf = (X,R,A)$, we define the binary relation $\R$ as follows:
\begin{center}
    $\R = \set{(x,y) \in X \times X : x=y \text{ or } Rxy \text{ or } Ryx}$.
\end{center}
Let $\R^0$ be the diagonal relation and we define $\R^n \coloneqq \R^{n-1}\circ\R$ for all $n \in \Zp$. Intuitively, $\R^n[x]$ is the set of points that $x$ can reach in $n$ steps. Correspondingly, we define a set of new operators $\is{n}$ and their duals $\all{n}$. These operators serve as ``$n$-step modalities'' and will play an important role in our proofs.

\begin{definition}
    For each $n\in\omega$ and $\phi,\psi\in\Lt$, we define the formula $\is{n}\phi$ by:
\begin{center}
    $\is{0}\phi=\phi$ and $\is{k+1}\phi=\is{k}\phi\vee\D\is{k}\phi\vee\bd\is{k}\phi$.
\end{center}
As usual, we define the dual operator $\all{n}$ of $\is{n}$ by $\all{n}\phi\coloneqq\neg\is{n}\neg\phi$. 
\end{definition}

\begin{proposition}
    Let $\mm=(X,R,V)$ be a model, $x\in X$ and $\phi\in\Lt$. Then for all $k\in\omega$,
    \begin{center}
        $\mm,x\md\is{k}\phi$ if and only if $\mm,y\md\phi$ for some $y\in\R^k[x]$.
    \end{center}
\end{proposition}

For every $n \in \Zp$, let 
\[\axiom{br}_n \coloneqq \is{n+1}p\to\is{n}p.\] 
For every rooted refined frame $\gf = (X,R,A)$, we have $\gf \md \axiom{br}_n$ iff $X = \R^{n}[x]$ for all $x \in X$.

We recall the following result from \cite{Chen2026a}, which will be used in Section~\ref{sec:S4t--chap:Undecidability}.

\begin{theorem}\label{thm:grz+bp}
    For all $k,l,n,m \in \Zp$, every extension of $\TL{Grz} \oplus \set{\axiom{bd}_k,\axiom{br}_l,\axiom{bw}_n,\axiom{bw}^\partial_m}$ is tabular.
\end{theorem}

\subsection{Decision problem of properties} \label{subsec2.2}

In this section, we briefly recall decision problems of properties. While we present the notions for tense logics, they apply to other non-classical logics as well, including modal logics in general, superintuitionistic logics, and bi-superintuitionistic logics. 
For the details and further references, we refer the reader to \cite[Chapter 17]{Chagrov.Zakharyaschev1997} and \cite{Wolter.Zakharyaschev2007}. 

Let $L_0$ be a tense logic, considered as the base logic. We identify a property $P$ in the lattice $\NExt{L_0}$ with the set of logics in $\NExt{L_0}$ that have the property $P$, that is, 
\[P = \{L \in \NExt{L_0}: L \text{ has } P\}.\] 
The informal description of decidability of logics we have used several times is formalized as the following definition.

\begin{definition} \label{def:decidable}
    Let $L_0$ be a tense logic. A property $P$ is \emph{decidable} in $\NExt{L_0}$ if and only if the set $\{\phi: L_0 \oplus \phi \in P\}$ is decidable.
\end{definition}

Following the convention, we restrict ourselves to finitely axiomatizable logics because an input for an algorithm must be a finite object. Technically, one could also consider recursively axiomatizable logics, with algorithms enumerating their axioms as inputs. However, Kuznetsov showed that only trivial properties can be decidable in this setting, similar to Rice's theorem (see, e.g., \cite[Section 17.1]{Chagrov.Zakharyaschev1997}). Thus, we only consider finitely axiomatizable logics, and use finite sets of formulas axiomatizing the logics, or equivalently, a single formula axiomatizing the logics as inputs. Since most logics we encounter in practice are finitely axiomatizable, this is not a serious drawback. Moreover, this formalization aligns with our usual expression of logics such as $\TL{K4} = \TL{K} \oplus \B \B p \to \B p$.

We note a direct observation of the above definition of decidability of properties. Determining a property in a larger lattice of logics is at least as hard as in a smaller one, in the following sense.

\begin{proposition} \label{prop:finite-extension}
    Let $L$ be a tense logic and $L'$ an extension of $L$ with finitely many axioms. If a property $P$ is undecidable in $\NExt{L'}$, then it is undecidable in $\NExt{L}$.
\end{proposition}

\begin{proof}
    We may assume $L' = L \oplus \phi$ for a formula $\phi$. We prove the contrapositive. Let $P$ be a decidable property in $\NExt{L}$. Then, given a formula $\psi$, we can determine whether $L' \oplus \psi$ has $P$ by asking whether $L \oplus (\phi \land \psi)$ has $P$ since the two logics are the same.
\end{proof}

\subsection{Minsky machines} \label{Subsec-Minsky}

In this section, we recall the basics of Minsky machines, which play an important role in our proofs. For more details, we refer to \cite[Sections 16 and 17]{Chagrov.Zakharyaschev1997} and \cite{minskyComputationFiniteInfinite1967}.
By a \emph{Minsky machine} we mean a finite set $\M$ of instructions involving finitely many \emph{states} and acting on two \textit{registers}, which are supposed to store one natural number each. A \emph{configuration} of a Minsky machine is a tuple $\tup{s, n, m}$, where $s$ is the current state of the machine, and $n$ and $m$ are the natural numbers in each register. An \textit{instruction} $I \in \M$ operates on the state and one of the two registers. More specifically, $I$ has one of the following four forms: 

\begin{itemize}
    \item $I = t \to \tup{t', 1, 0}$. In this case, $I$ turns the state $t$ into $t'$ and increments the first register;
    \item $I = t \to \tup{t', 0, 1}$. In this case, $I$ turns the state $t$ into $t'$ and increments the second register;
    \item $I = t \to \tup{t', -1, 0} (\tup{t'', 0, 0})$. In this case, $I$ turns the state $t$ into $t'$ and decrements the first register if the number in the first register is non-zero, and turns the state $t$ into $t''$ without modifying the registers otherwise;
    \item $I = t \to \tup{t',0, -1} (\tup{t'', 0, 0})$. In this case, $I$ turns the state $t$ into $t'$ and decrements the second register if the number in the second register is non-zero, and turns the state $t$ into $t''$ without modifying the registers otherwise.
\end{itemize}

For example, after applying $I = s \to \tup{s', -1, 0} (\tup{s'', 0, 0})$ to the configuration $\tup{s, n, m}$, we obtain the configuration $\tup{s', n-1, m}$ if $n \geq 1$ and the configuration $\tup{s'', n, m}$ if $n = 0$. 

In this paper, Minsky machines are assumed to be \emph{deterministic}, that is, for each state $t$ there is at most one instruction that acts on the state $t$. For a Minsky machine $\M$, we write $\M: \tup{s,n,m} \rightsquigarrow \tup{t,k,l}$ if, starting from the configuration $\tup{s,n,m}$, by applying the instructions in $\M$, we can reach the configuration $\tup{t,k,l}$ in finitely many (possibly 0) steps. For each configuration $\tup{s,n,m}$, let $\Reach_{\M}(\tup{s,n,m})$ denote the set
\begin{center}
    $\set{\tup{t,k,l}: (\M: \tup{s,n,m} \rightsquigarrow \tup{t,k,l})}$.
\end{center}
We call \emph{$\Reach_{\M}(\tup{s,n,m})$} the \emph{reachability set} of $\tup{s,n,m}$ in $\M$, and omit $\M$ when clear.

\begin{example}
    Let $\M = \set{s \to \tup{s, 1, 0}}$. Then, for any $n, m \in \omega$, 
    \[\Reach_{\M}(\tup{s,n,m}) = \set{\tup{s, n+i, m}: i \in \omega}.\]
\end{example}

We will use the following undecidable problem in our undecidability proofs, which is called the \emph{second configuration problem} in \cite[Theorem 16.3]{Chagrov.Zakharyaschev1997}. 

\begin{theorem} \label{Thm:undec-minsky}
    There exist a Minsky machine $\M$ and a configuration $\tup{s,n,m}$ such that the reachability from $\tup{s,n,m}$ in $\M$ is undecidable, that is, the set $\Reach_{\M}(\tup{s,n,m})$ is undecidable.
\end{theorem}

\section{Decidable Properties in \texorpdfstring{$\NExt{\TL{S4}}$}{NExt S4t} and \texorpdfstring{$\NExt{\TL{Grz}}$}{NExt Grzt}} \label{sec 3}

In this section, we study properties that are decidable in $\NExt{\TL{S4}}$ and $\NExt{\TL{Grz}}$. We prove that in the lattice $\NExt{\TL{S4}}$, consistency is decidable, and there are infinitely many tabular logics whose coincidence problem is decidable. It follows from Proposition~\ref{prop:finite-extension} that consistency is decidable in $\NExt{\TL{Grz}}$ as well. Moreover, we show that there is a tabular logic in $\NExt{\TL{Grz}}$ whose coincidence problem is decidable.

Let us begin with consistency in $\NExt\TL{S4}$. For each $n \in \Zp$, let $\Ch_n$ denote the $n$-chain, i.e., the frame $(\set{1, \dots, n}, \leq)$. For any logic $L = \TL{S4} \oplus \phi$, we have that $L$ is consistent if and only if $\Ch_1 \md \phi$ \cite{Kracht1992}.
Consequently, we obtain the following theorem.
\begin{theorem} \label{thm: consis S4t}
    Consistency is decidable in $\NExt\TL{S4}$.
\end{theorem}

\begin{proof}
    This follows directly from the above observation and the fact that the validity of a formula on a finite frame is decidable.
\end{proof}

Note that this contrasts the fact that consistency is undecidable in $\NExt\TL{K4}$ \cite{Chagrov.Shehtman1995,Chen.Takahashi2026}.

Next, we turn to coincidence problems. For each $n \in \Zp$, let $\Cl_{n}$ denote the $n$-cluster frame, i.e., the frame consisting of $n$ points with the full relation. We show that for all $n \in \Zp$, the coincidence problem for $\Log(\Cl_n)$ is decidable in $\NExt\TL{S4}$.

\begin{lemma}
    Let $n \in \Zp$. Then, coincidence with $\Log(\Cl_n)$ is decidable in $\NExt\TL{S4}$.
\end{lemma}
\begin{proof}
Take any $\phi\in\FormT$. Let $L_\phi \coloneqq \TL{S4}\oplus\phi$. We prove the following claim.

\begin{claim}
    $L_\phi=\Log(\Cl_n)$ iff the following holds:
    \begin{center}
        $(\dag)$  $\Ch_2\nmd\phi$, $\Cl_{n+1}\nmd\phi$, and $\Cl_n\md\phi$.
    \end{center}
\end{claim}
\begin{proof}
First, suppose that $\Ch_2\md\phi$. Then $L_\phi \sub \Log(\Ch_2)$. Recall from \cite{Kracht1992} that every logic $L\in\NExt{\TL{S4}}$ satisfies either $L\subseteq\Log(\Ch_2)$ or $L\supseteq\TL{S5}$. Thus, we obtain that $L_\phi \nsupseteq \TL{S5}$ and so $L_\phi\neq\Log(\Cl_n)$.
Suppose next that $\Cl_{n+1} \md\phi$. Then clearly, $L_\phi\sub \Log(\Cl_{n+1}) \subsetneq \Log(\Cl_{n})$. 

Finally, suppose that $\Ch_2\nmd\phi$ and $\Cl_{n+1}\nmd\phi$. Then $L_\phi \in \NExt\TL{S5}$ and $L_\phi \nsubseteq \Log(\Cl_{n+1})$. Recall from \cite[Theorem 5.4.1]{Chen2026a} that every $L\in\NExt{\TL{S5}}$ satisfies either $L\subseteq\Log(\Cl_{n+1})$ or $L\supseteq\Log(\Cl_n)$. Thus, we obtain that 
\begin{center}
    $\Cl_n \md \phi$ iff $L_\phi \sub \Log(\Cl_n)$ iff $L_\phi = \Log(\Cl_n)$.
\end{center}
This concludes our proof of the claim.
\end{proof}
It is clear that (\dag) is decidable, which concludes the proof.
\end{proof}

Consequently, we obtain the following theorem.

\begin{theorem} \label{thm:dec-grzt-=tab}
    There are infinitely many tabular tense logics $L \in \NExt \TL{S4}$ such that coincidence with $L$ is decidable in $\NExt \TL{S4}$.
\end{theorem}

Now we move to the case for $\TL{Grz}$. The following is a direct consequence of Theorem~\ref{thm: consis S4t} and Proposition~\ref{prop:finite-extension}.

\begin{theorem}
    Consistency is decidable in $\NExt\TL{Grz}$.
\end{theorem}

However, the proof of Theorem~\ref{thm:dec-grzt-=tab} does not work for $\NExt\TL{Grz}$. For $n>1$, the logic $\Log(\Cl_n)$ is no longer an extension of $\TL{Grz}$ since $\Cl_n$ is not antisymmetric for $n > 1$ (see Proposition~\ref{prop:grz}). In $\NExt\TL{Grz}$, however, we can still find a tabular logic with a decidable coincidence problem. To prove this, we need the following lemma:

\begin{lemma}\label{lem:splitting-grzt}
    For all $L \in \NExt\TL{Grz}$, either $L \supseteq \Log(\Cl_1)$ or $L \subseteq \Log(\Cl_2)$.
\end{lemma}
\begin{proof}
    Take any $L \in \NExt\TL{Grz}$. Note that $\TL{S4} \subseteq \TL{Grz} \subseteq L$. By \cite{Kracht1992}, we see that either $L \sub \Log(\Cl_2)$ or $L \supseteq \TL{S5}$. If $L \supseteq \TL{S5}$, then $L = L\oplus\axiom{grz}\oplus\axiom{grz}^\partial \supseteq \TL{S5}\oplus\axiom{grz}\oplus\axiom{grz}^\partial = \Log(\Cl_1)$, which concludes the proof.
\end{proof}

\begin{theorem} \label{thm:dec-grzt-=Cl1}
    Coincidence with $\Log(\Cl_1)$ is decidable in $\NExt\TL{Grz}$.
\end{theorem}
\begin{proof}
Take any $\phi\in\FormT$. Let $L_\phi \coloneqq \TL{Grz}\oplus\phi$. We prove the following claim.

\begin{claim}
    $L_\phi=\Log(\Cl_1)$ iff the following holds:
    \begin{center}
        $(\dag)$  $\Ch_2\nmd\phi$ and $\Cl_1\md\phi$.
    \end{center}
\end{claim}
\begin{proof}
Suppose that $\Ch_2\md\phi$. Then clearly we have $L_\phi\sub \Log(\Cl_{2}) \subsetneq \Log(\Cl_{1})$, so $L_\phi\neq\Log(\Cl_1)$. Suppose that $\Ch_2\nmd\phi$. Then $L_\phi \nsubseteq \Log(\Cl_{2})$. By Lemma~\ref{lem:splitting-grzt}, $L_\phi \supseteq \Log(\Cl_1)$. Thus,
\begin{center}
    $\Cl_1 \md \phi$ iff $L_\phi \sub \Log(\Cl_1)$ iff $L_\phi = \Log(\Cl_1)$.
\end{center}
This concludes our proof of the claim.
\end{proof}
It is clear that (\dag) is decidable, which concludes the proof.
\end{proof}

\begin{remark}
    The above proofs use the idea of \emph{splittings}. We refer to \cite{Chen2026a,Kracht1992} for details of the notion of splittings and its applications in tense logics. For example, it was shown that coincidence with any \emph{iterated splitting} is decidable in $\NExt\TL{S4}$ \cite[Theorem 6.3.2]{Chen2026a}.
\end{remark}

\section{Undecidable Properties in \texorpdfstring{$\NExt\TL{Grz}$}{NExt Grzt}} \label{sec:S4t--chap:Undecidability} 

The aim of this section is to study undecidable properties in $\NExt\TL{Grz}$. Following Chagrov's method \cite{chagrov1990I, chagrov1990II, Chagrov.Zakharyaschev1993, Chagrov.Chagrova1995, Chagrov2002}, we reduce the undecidable problem about Minsky machines in Theorem~\ref{Thm:undec-minsky} to the decision problem of properties in $\NExt\TL{Grz}$. Although this method was recently applied to $\NExt{\TL{K4}}$ \cite{Chen.Takahashi2026}, adapting it to $\NExt\TL{Grz}$ is not straightforward: the main difficulty is that frames are now reflexive and thus we cannot use variable-free formulas to define subsets uniformly for all valuations, which makes it harder to simulate Minsky machines.

To overcome this difficulty, we isolate a trick that has been used in the literature but has remained largely implicit in proofs, and formulate it explicitly as the method of \emph{good valuations} (see Definition~\ref{def-good-val}). The use of good valuations is not limited to undecidability proofs based on Chagrov's method, such as \cite{Chagrov.Zakharyaschev1993}. The method also appears in Fine's proof of a Kripke incomplete logic above $\ML{S4}$ \cite{Fine1974IncompleteLogicContainingS4}, and in the proof of Blok's dichotomy theorem (originally proved in \cite{Blok1978degree}) given in \cite[Theorem 10.59]{Chagrov.Zakharyaschev1997}. Note that the part of the latter proof requiring good valuations is omitted, but the underlying method is described in \cite[Example 10.58]{Chagrov.Zakharyaschev1997}.

\subsection{Main theorems}\label{subsec:main theorems}

First, we state the two main theorems and discuss their implications. We only summarize the proof ideas at the end of this subsection, postponing the detailed proofs to the subsequent subsections.
The first main theorem provides a general criterion for a property to be undecidable in $\NExt \TL{Grz}$.

\begin{theorem}[Main Theorem 1] \label{thm:undec-scheme-grzt}
    Let $P$ be a property satisfying the following two conditions. Then, $P$ is undecidable in $\NExt \TL{Grz}$.
    \begin{enumerate}
        \item The logic $\TL{Grz} \oplus \set{\axiom{bd}_2, \axiom{bw}_2, \axiom{bw}^\partial_5,\axiom{br}_7}$ has $P$,
        \item For any logic $L \in \NExt \TL{Grz}$, if $L$ has $P$, then $L$ is Kripke complete or decidable.
    \end{enumerate}
\end{theorem}

Note that the logic $\TL{Grz} \oplus \set{\axiom{bd}_2, \axiom{bw}_2, \axiom{bw}^\partial_5,\axiom{br}_7}$ is tabular by Theorem~\ref{thm:grz+bp}. Thus, condition (1) is satisfied for any property $P$ that holds for all tabular logics. In fact, most properties one studies for logics fall into the scope of this theorem, as shown in the following corollary.

\begin{corollary} \label{cor:undec-prop-grzt}
    The following properties are undecidable in $\NExt \TL{Grz}$:
    \begin{enumerate}
        \item tabularity,
        \item local tabularity,
        \item the FMP,
        \item decidability,
        \item elementarity,
        \item canonicity,
        \item Kripke completeness.
    \end{enumerate}
\end{corollary}

\begin{proof}
    We check the two conditions in Theorem~\ref{thm:undec-scheme-grzt} for the listed properties. For the second condition, it is clear that (4) implies decidability, and all other properties imply Kripke completeness. For the first condition, it suffices to observe that, by Theorem~\ref{thm:grz+bp}, the logic $\TL{Grz} \oplus \set{\axiom{bd}_2, \axiom{bw}_2, \axiom{bw}^\partial_5,\axiom{br}_7}$ is tabular. So, it satisfies all the listed properties. Thus, by Theorem~\ref{thm:undec-scheme-grzt}, we conclude that these properties are undecidable in $\NExt \TL{Grz}$.
\end{proof}

It follows from Proposition~\ref{prop:finite-extension} that these properties are undecidable in any lattice of extensions of a base logic over which $\TL{Grz}$ is finitely axiomatizable. In particular, they are undecidable in $\NExt \TL{S4}$, $\NExt \TL{K4}$, and $\NExt \TL{K}$, generalizing the undecidability results in  \cite{Chagrov.Shehtman1995, Chen.Takahashi2026}.

\begin{corollary} \label{cor:undec-prop-S4t}
    The properties listed in Corollary~\ref{cor:undec-prop-grzt} are undecidable in the lattices $\NExt \TL{S4}$, $\NExt \TL{K4}$, and $\NExt \TL{K}$. 
\end{corollary}

\begin{proof}
    This follows immediately from Corollary~\ref{cor:undec-prop-grzt} and Proposition~\ref{prop:finite-extension}.
\end{proof}

Another consequence concerns bi-superintuitionistic logics. Recall that \emph{superintuitionistic logics} are extensions of the intuitionistic logic $\IL{IPC}$, and \emph{bi-superintuitionistic logics} are extensions of the bi-intuitionistic logic $\IL{biIPC}$, also known as \emph{Heyting-Brouwer logic}, whose language has an additional symbol $\leftarrow$ of \emph{coimplication} \cite{Rauszer1974, wolterLogicsCoimplication1998}. The celebrated Blok-Esakia theorem states that there is a lattice isomorphism between $\NExt \ML{Grz}$ and $\Ext \IL{IPC}$, the lattice of superintuitionistic logics \cite{Blokthesis, Esakia1976ModalCompanions}. Many properties are preserved and reflected under this isomorphism (see, e.g., \cite[Section 9.6]{Chagrov.Zakharyaschev1997}), which allows us to transfer undecidability results between $\NExt \ML{Grz}$ and $\Ext \IL{IPC}$. This was demonstrated in \cite{Chagrov.Zakharyaschev1993}. Wolter \cite{wolterLogicsCoimplication1998} proved an analogue of Blok-Esakia theorem for $\NExt \TL{Grz}$ and $\Ext \IL{biIPC}$, and showed that the FMP is preserved and reflected under the corresponding isomorphism. We use this fact to transfer the undecidability of some properties from $\NExt \TL{Grz}$ to $\Ext \IL{biIPC}$, in the same spirit of \cite{Chagrov.Zakharyaschev1993}. See also \cite{BezhanishviliBLOKESAKIATHEOREMSSTABLE2025a} for a uniform treatment of these Blok-Esakia theorems. 

Let us recall the minimal preliminaries on the Blok-Esakia theorem for $\NExt \TL{Grz}$ and $\Ext \IL{biIPC}$; see \cite{wolterLogicsCoimplication1998,BezhanishviliBLOKESAKIATHEOREMSSTABLE2025a} for details. Let $t$ be the extension of G\"odel-McKinsey-Tarski translation to the bi-intuitionistic setting, i.e., $t$ translates a bi-intuitionistic formula into a tense formula. In particular, $t(\phi\to\psi) \coloneqq \B(t(\phi) \to t(\psi))$ and $t(\phi\leftarrow\psi) \coloneqq \bd(t(\phi) \wedge \neg t(\psi))$. The maps $\sigma: \Ext \IL{biIPC} \to \NExt \TL{Grz}$ and $\rho: \NExt \TL{Grz} \to \Ext \IL{biIPC}$ are defined as 
\[\sigma L = \TL{Grz} \oplus \set{t(\phi): \phi \in L}, \quad \rho \ML{M} = \IL{biIPC} \oplus \set{\phi: t(\phi) \in \ML{M}}.\]
The two maps $\sigma$ and $\rho$ are mutually inverse lattice isomorphisms.

For our purposes, we also need to compute these maps on finitely axiomatizable logics. 

\begin{lemma} \label{lem:computable-BE}
    The maps $\sigma: \Ext \IL{biIPC} \to \NExt \TL{Grz}$ and $\rho: \NExt \TL{Grz} \to \Ext \IL{biIPC}$ are computable on finitely axiomatizable logics. That is, these isomorphisms map a finitely axiomatizable logic to a finitely axiomatizable one, and there is an algorithm that, given a bi-intuitionistic (resp. tense) formula $\psi$, produces an axiom of $\sigma(\IL{biIPC} \oplus \psi)$ (resp. $\rho(\TL{Grz} \oplus \psi)$).
\end{lemma}

\begin{proof}
    The statement for $\sigma$ follows immediately from its definition, by noting that the syntactic translation $t$ is computable. 
    
    We show the statement for $\rho$. Let $\ML{M} = \TL{Grz} \oplus \psi$. Then, since $\ML{M} = \sigma \rho \ML{M}$, we have $\TL{Grz} \oplus \psi = \TL{Grz} \oplus \set{t(\phi): t(\phi) \in \ML{M}}$. By the compactness and the fact that $t$ commutes with conjunction, there exists a bi-intuitionistic formula $\phi$ such that $\TL{Grz} \oplus \psi = \TL{Grz} \oplus t(\phi)$, hence $\rho \ML{M} = \IL{biIPC} \oplus \phi$. 

    To find such a $\phi$, we can enumerate all bi-intuitionistic formulas $\phi_0, \phi_1, \ldots$, while trying to verify $\TL{Grz} \oplus \psi = \TL{Grz} \oplus t(\phi_i)$ for each $\phi_i$. The existence of a formula $\phi_i$ with $\TL{Grz} \oplus \psi = \TL{Grz} \oplus t(\phi_i)$ is guaranteed by the previous argument. Moreover, for such a formula $\phi_i$, since both $\TL{Grz} \oplus \psi$ and $\TL{Grz} \oplus t(\phi_i)$ are recursively enumerable by enumerating all proofs, we can verify $t(\phi_i) \in \TL{Grz} \oplus \psi$ and $\psi \in \TL{Grz} \oplus t(\phi_i)$ in finite time if they are true. Thus, this algorithm will eventually find a desired formula $\phi$ and verify $\TL{Grz} \oplus \psi = \TL{Grz} \oplus t(\phi)$. Therefore, $\rho$ is computable on finitely axiomatizable logics. 
\end{proof}

\begin{corollary} \label{cor:undec-fmp-biIPC}
    The FMP and tabularity are undecidable in $\Ext \IL{biIPC}$. 
\end{corollary}

\begin{proof}
   It suffices to observe that the map $\rho: \NExt \TL{Grz} \to \Ext \IL{biIPC}$ restricted to finitely axiomatizable logics, which is computable by Lemma~\ref{lem:computable-BE}, is the desired reduction from the decision problem of each of the properties in $\NExt \TL{Grz}$ to that in $\Ext \IL{biIPC}$.

    The two maps $\sigma$ and $\rho$ are mutually inverse lattice isomorphisms which preserve and reflect the FMP and tabularity \cite{wolterLogicsCoimplication1998}. It follows that for any $\ML{M} \in \NExt \TL{Grz}$, we have that $\rho \ML{M}$ enjoys the FMP (resp. is tabular) if and only if $\ML{M} = \sigma \rho \ML{M}$ enjoys the FMP (resp. is tabular). Thus, both the FMP and tabularity are undecidable in $\Ext \IL{biIPC}$, since they are undecidable in $\NExt \TL{Grz}$ by Corollary~\ref{cor:undec-prop-grzt}.
\end{proof}

\begin{remark}
    As pointed out in \cite{WolterZakharyaschev2014}, ``in contrast to the situation for si-logics [superintuitionistic logics], the preservation properties of those mappings [in the setting of $\NExt \TL{Grz}$ and $\Ext \IL{biIPC}$] have not yet been investigated in any detail.'' We leave it for future work to systematically study preservation theorems for the Blok-Esakia theorem for $\NExt \TL{Grz}$ and $\Ext \IL{biIPC}$, and thus transfer our undecidability results in $\NExt \TL{Grz}$ to $\Ext \IL{biIPC}$ beyond the FMP and tabularity.
\end{remark}

As for the decision problem of coincidence with a tabular logic, we have shown in Theorems~\ref{thm:dec-grzt-=tab} and \ref{thm:dec-grzt-=Cl1} that, contrary to the case of $\NExt \TL{K4}$, both $\NExt \TL{S4}$ and $\NExt \TL{Grz}$ contain a tabular logic with a decidable coincidence problem. So, for $\NExt \TL{S4}$ and $\NExt \TL{Grz}$, there is no general undecidability result for the coincidence problems as in the case of $\NExt \TL{K4}$ \cite{Chen.Takahashi2026}. However, we still have the following theorem, which states that there are as many tabular logics in $\NExt \TL{Grz}$ as there could be with an undecidable coincidence problem.

\begin{theorem}[Main theorem 2] \label{thm:undec-prop-S4t-=tab}
    There are infinitely many tabular tense logics $L \in \NExt \TL{Grz}$ such that coincidence with $L$ is undecidable in $\NExt \TL{Grz}$.
\end{theorem}

\begin{corollary} \label{cor:undec-prop-S4t-=tab}
    There are infinitely many tabular tense logics $L \in \NExt \TL{S4}$ such that coincidence with $L$ is undecidable in $\NExt \TL{S4}$.
\end{corollary}

\begin{proof}
    This follows immediately from Theorem~\ref{thm:undec-prop-S4t-=tab}, Proposition~\ref{prop:finite-extension}, and the fact that any tabular logic in $\NExt \TL{Grz}$ is in $\NExt \TL{S4}$.
\end{proof}

We have consequences on $\Ext \IL{biIPC}$ similar to Corollary~\ref{cor:undec-fmp-biIPC}.

\begin{corollary} \label{cor:undec-prop-biIPC-=tab}
    There are infinitely many tabular bi-superintuitionistic logics $L \in \Ext \IL{biIPC}$ such that coincidence with $L$ is undecidable in $\Ext \IL{biIPC}$.
\end{corollary}

\begin{proof} 
    This follows from Theorem~\ref{thm:undec-prop-S4t-=tab} and Lemma~\ref{lem:computable-BE} by a similar proof as Corollary~\ref{cor:undec-fmp-biIPC}.
\end{proof}

\begin{remark}
    Since there are countably many tabular logics in total, Theorem~\ref{thm:dec-grzt-=tab} and Corollary~\ref{cor:undec-prop-S4t-=tab} together imply that there are as many tabular logics in $\NExt \TL{S4}$ as there could be with a decidable and an undecidable coincidence problem. However, we do not have a full characterization of tabular logics $L$ in $\NExt \TL{S4}$ for which the coincidence problem is decidable. We also leave it open whether there is more than one consistent tabular logic $L$ in $\NExt \TL{Grz}$ or $\Ext \IL{biIPC}$ with a decidable coincidence problem. Note that one such logic in $\NExt \TL{Grz}$ is given by Theorem~\ref{thm:dec-grzt-=Cl1}, and one in $\Ext \IL{biIPC}$ follows from \cite[Theorem 31]{wolterLogicsCoimplication1998}.
\end{remark}

This concludes our discussion of the main theorems and their implications. We close this subsection by giving a brief outline of the proof of the main theorems, Theorems~\ref{thm:undec-scheme-grzt} and \ref{thm:undec-prop-S4t-=tab}. The rest of the paper is devoted to proving them. 

According to Theorem~\ref{Thm:undec-minsky}, there exist a Minsky machine $\M$ and a configuration $\tup{s,n,m}$ of $\M$ such that the set $\Reach_{\M}(\tup{s,n,m}) = \set{\tup{t,k,l}: (\M: \tup{s,n,m}\rightsquigarrow\tup{t,k,l})}$ is undecidable. In what follows, we fix such a Minsky machine $\M$ and a configuration $\tup{s,n,m}$ of $\M$. To simplify notation, we write $\Reach$ for the set $\Reach_{\M}(\tup{s,n,m})$.

The proof proceeds as follows.
\begin{enumerate}
    \item Section~\ref{subsec-1}: We construct a general frame $\gf$ that encodes the undecidable set $\Reach$ of configurations. We define a formula $\phi^*$ that characterizes \emph{good} valuations (see Definition~\ref{def-good-val}) on $\gf$, under which enough subsets of $\gf$ are definable.
    \item Section~\ref{subsec-2}: We define a formula $AxM$ that encodes the behavior of $\M$. By observing how $\gf$ simulates computation in $\M$, we show that $\gf \models AxM$.
    \item Section~\ref{subsec-3}: We construct a reduction that maps a configuration $\tup{t,k,l}$ to a logic $L(t, k, l)$ in $\NExt \TL{Grz}$ that yields Theorem~\ref{thm:undec-scheme-grzt}. We show Theorem~\ref{thm:undec-prop-S4t-=tab} by slightly modifying the construction for each $d \in \omega$.
\end{enumerate}

\subsection{The general frame \texorpdfstring{$\gf$}{F} and good valuations} \label{subsec-1}

Let us begin the proof by defining the general frame $\gf$, which encodes the undecidable set $\Reach$ of configurations. It will be used, in particular, to show Kripke incompleteness and undecidability. Recall that we fixed a Minsky machine $\M$ and a configuration $\tup{s,n,m}$ of $\M$ such that the set $\Reach = \set{\tup{t,k,l}: (\M: \tup{s,n,m}\rightsquigarrow\tup{t,k,l})}$ is undecidable (cf. Theorem~\ref{Thm:undec-minsky}). We may assume that $\M$ contains $t_0$ many states, labeled as $0, \dots, t_0-1$.

\begin{definition}
    Let $\gf=(W,R,A)$ be the general frame defined as follows. See Figure~\ref{fig:gf-S4t} for an illustration of $\gf$.
    \begin{itemize}
        \item $W$ is the union of the following sets:
        \begin{itemize}
            \item $\set{a_i,b_i : i < \omega} (\eqqcolon AB)$,
            \item $\set{c_i,c'_i : i \leq t_0} (\eqqcolon C)$,
            \item $\set{x_i : i \leq 5}$,
            \item $\set{e}$,
            \item $\set{\tup{t,k,l}_i : i \leq 3 \text{ and } \tup{t,k,l} \in \Reach}$.
        \end{itemize}
        \item $R$ is the reflexive-transitive closure of the union of the following sets:
        \begin{itemize}
            \item $\set{(a_{i+1},a_i),(a_{i+1},b_i) : i \in \omega} \cup \set{(b_{i+1},b_i),(b_{i+2},a_i) : i \in \omega}$,
            \item $\set{(c_{i+1},c_i),(c'_{i+1},c'_i) : i < t_0} \cup \set{(c'_i,c_i),(c'_i,a_{i+1}) : i \leq t_0}$,
            \item $\set{(c_{t+1},\tup{t,k,l}_0),(a_{k+1},\tup{t,k,l}_1),(b_{l+1},\tup{t,k,l}_2) : \tup{t,k,l} \in \Reach}$,
            \item $\set{(z_1,z_0),(z_2,z_0),(z_2,z_3) : z \in \Reach}$,
            \item $\set{(x_0,x_1),(c_0,x_1),(x_2,x_1),(x_2,x_3),(x_4,x_3),(x_4,a_0),(x_5,b_0)}$,
            \item $\set{(e,b_i) : i \in \omega} \cup \set{(e,c'_{t_0})}$.
        \end{itemize}
        \item $A = \set{U\sub W: U \cap AB \text{ is finite or cofinite}}$.
    \end{itemize}
\end{definition}

\begin{figure}[htbp]
        \small
    \[
    \begin{tikzpicture}[scale=0.9]
    \def\ptRad{.15pt}
    \node (a0) at (0,8)[label=left:$a_0$]{$\circ$};
    \node (a1) at (0,7)[label=left:$a_1$]{$\circ$};
    \node (a2) at (0,6)[label=left:$a_2$]{$\circ$};
    \node (a3) at (0,5)[label=left:$a_3$]{$\circ$};
    \node (a4) at (0,4)[label=left:$a_4$]{$\circ$};
    \node (b0) at (1.5,8)[label=right:$b_0$]{$\circ$};
    \node (b1) at (1.5,7)[label=right:$b_1$]{$\circ$};
    \node (b2) at (1.50,6)[label=right:$b_2$]{$\circ$};
    \node (b3) at (1.5,5)[label=right:$b_3$]{$\circ$};
    \node (b4) at (1.5,4)[label=right:$b_4$]{$\circ$};   
    \draw [->] (a1) -- (a0);
    \draw [->] (a2) -- (a1);
    \draw [->] (a3) -- (a2);
    \draw [->] (a4) -- (a3);
    \draw [->] (b1) -- (b0);
    \draw [->] (b2) -- (b1);
    \draw [->] (b3) -- (b2);
    \draw [->] (b4) -- (b3);
    \draw [->] (a1) -- (b0);
    \draw [->] (a2) -- (b1);
    \draw [->] (a3) -- (b2);
    \draw [->] (a4) -- (b3);
    \draw [->] (b2) -- (a0);
    \draw [->] (b3) -- (a1);
    \draw [->] (b4) -- (a2);
    \node (vd1) at (.75,3.4){$\vdots$};
    \node (vd2) at (0,3.4){$\vdots$};
    \node (vd3) at (1.5,3.4){$\vdots$};

    \node (al0) at (0,2.5)[label=left:$a_{k}$]{$\circ$};
    \node (al1) at (0,1.5)[label=left:$a_{k+1}$]{$\circ$};
    \node (al2) at (0,0.5)[label=left:$a_{k+2}$]{$\circ$};
    \node (bl0) at (1.5,2.5)[label=right:$b_{k}$]{$\circ$};
    \node (bl1) at (1.5,1.5)[label=right:$b_{k+1}$]{$\circ$};
    \node (bl2) at (1.50,0.5)[label=right:$b_{k+2}$]{$\circ$};
    \draw [->] (al1) -- (al0);
    \draw [->] (al2) -- (al1);
    \draw [->] (bl1) -- (bl0);
    \draw [->] (bl2) -- (bl1);
    \draw [->] (al1) -- (bl0);
    \draw [->] (al2) -- (bl1);
    \draw [->] (bl2) -- (al0);
    \node (vd4) at (0,0){$\vdots$};
    \node (vd5) at (.75,0){$\vdots$};
    \node (vd6) at (1.5,0){$\vdots$};

    \node (ak0) at (0,-3.5+2.5)[label=left:$a_{l}$]{$\circ$};
    \node (ak1) at (0,-3.5+1.5)[label=left:$a_{l+1}$]{$\circ$};
    \node (ak2) at (0,-3.5+0.5)[label=left:$a_{l+2}$]{$\circ$};
    \node (bk0) at (1.5,-3.5+2.5)[label=right:$b_{l}$]{$\circ$};
    \node (bk1) at (1.5,-3.5+1.5)[label=right:$b_{l+1}$]{$\circ$};
    \node (bk2) at (1.50,-3.5+0.5)[label=right:$b_{l+2}$]{$\circ$};
    \draw [->] (ak1) -- (ak0);
    \draw [->] (ak2) -- (ak1);
    \draw [->] (bk1) -- (bk0);
    \draw [->] (bk2) -- (bk1);
    \draw [->] (ak1) -- (bk0);
    \draw [->] (ak2) -- (bk1);
    \draw [->] (bk2) -- (ak0);
    \node (vd4) at (0,-3.5+0){$\vdots$};
    \node (vd5) at (.75,-3.5+0){$\vdots$};
    \node (vd6) at (1.5,-3.5+0){$\vdots$};

    \node (x0) at (-5,7)[label=left:$x_0$]{$\circ$};
    \node (x1) at (-4,8)[label=left:$x_1$]{$\circ$};
    \node (x2) at (-3,7)[label=left:$x_2$]{$\circ$};
    \node (x3) at (-2,8)[label=left:$x_3$]{$\circ$};
    \node (x4) at (-1,7)[label=left:$x_4$]{$\circ$};
    \node (y1) at (2.5,7)[label=right:$x_5$]{$\circ$};
    \draw [->] (x0) -- (x1);
    \draw [->] (x2) -- (x1);
    \draw [->] (x2) -- (x3);
    \draw [->] (x4) -- (x3);
    \draw [->] (x4) -- (a0);
    \draw [->] (y1) -- (b0);

    \node (c0) at (-5+1,6)[label=left:$c_0$]{$\circ$};
    \node (c1) at (-5+1,-1+6)[label=left:$c_1$]{$\circ$};
    \node (vd) at (-5+1,-1+5.5){$\vdots$};
    \node (ct0) at (-5+1,-1+4.5)[label=left:$c_{t}$]{$\circ$};
    \node (ct1) at (-5+1,-1+3.5)[label=left:$c_{t+1}$]{$\circ$};
    \node (ci0) at (-5+1,-1+2)[label=left:$c_{t_0-1}$]{$\circ$};
    \node (ci1) at (-5+1,-1+1)[label=left:$c_{t_0}$]{$\circ$};
    \node (vd5) at (-5+1,-1+3){$\vdots$};    
    \draw [->] (c0) -- (x1);
    \draw [->] (c1) -- (c0);
    \draw [->] (ct1) -- (ct0);
    \draw [->] (ci1) -- (ci0);

    \node (c01) at (-4+1,-2+7)[label=right:$c_0'$]{$\circ$};
    \node (c11) at (-4+1,-2+6)[label=right:$c_1'$]{$\circ$};
    \node (vd1) at (-4+1,-2+5.5){$\vdots$};
    \node (ct01) at (-4+1,-2+4.5)[label=right:$c'_{t}$]{$\circ$};
    \node (ct11) at (-4+1,-2+3.5)[label=right:$c'_{t+1}$]{$\circ$};
    \node (ci01) at (-4+1,-2+2)[label=right:$c'_{t_0-1}$]{$\circ$};
    \node (ci11) at (-4+1,-2+1)[label=right:$c'_{t_0}$]{$\circ$};
    \node (vd51) at (-4+1,-2+3){$\vdots$};    
    \draw [->] (c01) -- (c0);
    \draw [->] (c11) -- (c1);
    \draw [->] (ct01) -- (ct0);
    \draw [->] (ct11) -- (ct1);
    \draw [->] (ci01) -- (ci0);
    \draw [->] (ci11) -- (ci1);
    \draw [->] (c01) -- (a1);
    \draw [->] (c11) -- (a2);
    \draw [->] (c11) -- (c01);
    \draw [->] (ct11) -- (ct01);
    \draw [->] (ci11) -- (ci01);
    \node (vd52) at (-2,4.5){$\vdots$};

    \node (tkl) at (5,5.5)[label=above:$\tup{t,k,l}_0$]{$\circ$};
    \node (tkl1) at (4,4.5)[label=below:$\tup{t,k,l}_1$]{$\circ$};
    \node (tkl2) at (6,4.5)[label=below:$\tup{t,k,l}_2$]{$\circ$};
    \node (tkl3) at (7,5.5)[label=above:$\tup{t,k,l}_3$]{$\circ$};
    \draw [->] (tkl1) -- (tkl);
    \draw [->] (tkl2) -- (tkl);
    \draw [->] (tkl2) -- (tkl3);
    \draw [->] (ct1) -- (tkl);
    \draw [->] (bk1) -- (tkl2);
    \draw [->] (al1) -- (tkl1);
    \node (e1) at (0,-4){};
    \node (e2) at (1.5,-4){};

    \node (end) at (-1,-5)[label=right:$e$]{$\circ$};
    \draw [->] (end) -- (e1);
    \draw [->] (end) -- (e2);
    \draw [->] (end) -- (ci11);

    \end{tikzpicture}
    \]
    \caption{The frame $\gf$}
    \label{fig:gf-S4t}
\end{figure}

The reader can readily check that $A$ is closed under Boolean operators, $R[\cdot]$ and $\inver{R}[\cdot]$, and so $\gf$ is well-defined. In fact, $\gf$ is a refined frame. 

Intuitively, $AB$ encodes the two counters of $\M$, $C$ encodes the states of $\M$, and points of the form $\tup{t,k,l}_i$ encode the configurations reachable from $\tup{s,n,m}$. The points $x_0, \dots, x_5$ are used to distinguish top points.

We note that every two points in $\gf$ are connected by an $\R$-path of length $\leq 7$. Thus, $\gf \md \axiom{br}_7$. This allows us to define so-called \emph{universal modalities} $\E$ and $\A$ on $\gf$ as follows. 

\begin{definition}
    For every formula $\phi$, let $\E\phi \coloneqq \is{7}\phi$ and $\A\phi \coloneqq \all{7}\phi$.
\end{definition}

Intuitively, $\gf, x \md \E \phi$ (resp. $\gf, x \md \A \phi$) iff $\gf, y \md \phi$ for some (resp. all) $y$. These modalities behave as desired on all rooted frames which validate $\axiom{br}_7$.

To simulate the computation of $\M$ in $\gf$, it is important to define subsets of $\gf$, in particular, singletons, by formulas uniformly for valuations. However, since $\gf$ is reflexive, it is impossible to take all valuations into account. Indeed, any variable-free formula is true everywhere or nowhere in $\gf$. Thus, we restrict to \emph{good} valuations, induced by a specific formula $\phi^*$. We will define subsets of $\gf$ by formulas uniformly for good valuations. In other words, many subsets of $\gf$ turn out to be uniformly definable ``relative to $\phi^*$''. 

\begin{definition} \label{def-good-val}
    Let $\phi^* \coloneq \lnot \axiom{bd}_2 \lor \lnot \axiom{bw}_2$, where
    \begin{align*}
        \axiom{bd}_2 &= \D(\B q_1 \wedge\neg (\D\B q_0\to q_0))\to q_1, \\
        \axiom{bw}_2 &= \bigwedge_{i\leq 2}\D p_i\to\bigvee_{i\neq j\leq 2}\D(p_i\wedge(p_j\vee\D p_j)). 
    \end{align*}
    A valuation $V$ on $\gf$ is called \emph{good} if $V(\phi^*) \neq \ve$.
\end{definition}

See Section~\ref{sec 2.1} for the meaning of the formulas $\axiom{bd}_2$ and $\axiom{bw}_2$. Now we introduce some formulas and show that they define subsets in $\gf$ uniformly for good valuations.

\begin{definition} \label{def:S4t-char-formulas-1}
    We define the following formulas:
    \begin{align*}
        \phi_{x_3} &\coloneq \B\bb\B\neg \phi^*, && \phi_{x_5} \coloneq \bb\B\bb\B\bb\neg\phi_{x_3}, && \phi_{x_0} \coloneq \B\neg\phi_{x_3} \wedge \bb\B\bb\neg\phi_{x_5},\\
        \phi_{x_1} &\coloneq \bd\phi_{x_0} \wedge \neg\phi_{x_0}, && \phi_{x_2} \coloneq \D\phi_{x_1} \wedge \D\phi_{x_3}, && \phi_{x_4} \coloneq \D\phi_{x_3} \wedge \bb\neg\phi_{x_2},\\
        \phi_{a_0} &\coloneq \bd\phi_{x_4} \wedge \B\neg\phi_{x_3}, && \phi_{b_0} \coloneq \bd\phi_{x_5} \wedge \neg\phi_{x_5}, && \phi_{b_1} \coloneq \D\phi_{b_0} \wedge \B\neg\phi_{a_0} \wedge \bb\neg\phi_{x_5},\\
        \phi_{AB} & \coloneq \mathrlap{\B\neg\phi_{x_1} \wedge \B\neg\phi_{x_5} \wedge (\D\phi_{b_0} \vee \phi_{a_0}).} 
    \end{align*}
\end{definition}


\begin{lemma} \label{lem:S4t-formula-location-1}
    Let $V$ be a good valuation on $\gf$. Then, the following hold:
    \begin{enumerate}
        \item $V(\phi_{u}) = \set{u}$ for each $u \in \set{x_i : i \leq 5} \cup \set{a_0,b_0,b_1}$;
        \item $V(\phi_{AB}) = AB = \set{a_i,b_i:i\in\omega}$.
    \end{enumerate}
\end{lemma}
\begin{proof}
    We begin by observing that $V(\phi^*) \cap (\set{x_i : i \leq 5} \cup \set{a_0,b_0}) = \ve$. Indeed, $\phi^*$ can be satisfied only at points $w$ such that $R[w]$ contains either a strict chain or an anti-chain of size greater than $2$, which holds for no point in $\set{x_i : i \leq 5} \cup \set{a_0,b_0}$.

    Let $V$ be a good valuation on $\gf$. By the above observation, we see that $e \in \inver{R}[V(\phi^*)]$ and $\inver{R}[V(\phi^*)] \cap (\set{x_i : i \leq 5} \cup \set{a_0,b_0}) = \ve$. Thus, $R[\inver{R}[V(\phi^*)]] = W \setminus \set{x_0,x_2,x_3,x_4,x_5}$, and $\inver{R}[R[\inver{R}[V(\phi^*)]]] = W \setminus \set{x_3}$. It follows that $V(\phi_{x_3}) = \set{x_3}$. The rest of the proof can be done similarly by unfolding the definitions of the formulas.
\end{proof}

The following lemma shows that there exist good valuations on $\gf$.

\begin{lemma} \label{lem:S4t-gf-refutes-phi*}
    $\gf \nmd \lnot \phi^*$.
\end{lemma}
\begin{proof}
    Let $V$ be a valuation on $\gf$ such that $V(q_1) = \set{a_0,a_1,b_0}$ and $V(q_0) = \set{a_0}$. Then, we see that $\gf,V,a_2 \not\md \axiom{bd}_2$, which entails $\gf,V,a_2 \md \phi^*$, and so $\gf \nmd \lnot \phi^*$.
\end{proof}


\subsection{Minsky machine simulation} \label{subsec-2}

This is the key part of the proof. We will define formulas that simulate the computation of $\M$ not only syntactically, but also semantically on $\gf$. 

The infinite ``ladder'' $AB$ will play an important role. Let $\phi_{AB}[q'/\phi^*]$ be the formula obtained from $\phi_{AB}$ by replacing $\phi^*$ with $q'$. In other words, $\phi_{AB}[q'/\phi^*]$ is defined the same way as in Definition~\ref{def:S4t-char-formulas-1} but starting with $\phi_{x_3}[q'/\phi^*] = \B\bb\B\neg q'$. Recall that in the proof of Lemma~\ref{lem:S4t-formula-location-1}, we only used the assumption that $V(\phi^*) \neq \ve$ and the observation that $V(\phi^*) \cap (\set{x_i : i \leq 5} \cup \set{a_0,b_0}) = \ve$. Thus, we obtain a similar lemma for $\phi_{AB}[q'/\phi^*]$. 

\begin{lemma} \label{lem:S4t-formula-location-phiABq'}
    Let $V$ be a valuation on $\gf$. Suppose $V(q') \neq \ve$ and $V(q') \cap (\set{x_i : i \leq 5} \cup \set{a_0,b_0}) = \ve$. Then, $V(\phi_{AB}[q'/\phi^*]) = \set{a_i,b_i:i\in\omega}$.
\end{lemma}

We will use ``triangles'' in $AB$, such as $\set{a_i, b_i, b_{i+1}}$, to represent numbers in the two counters of $\M$. The three formulas in the following definition move such a triangle one step below, simulating the increment of a counter. See Figure~\ref{fig:formula-tau} for an illustration. While Lemma~\ref{lem:formula-tau}, as shown in Figure~\ref{fig:formula-tau}, concerns two types of triangles, only the first type will appear in our simulation of $\M$.

\begin{definition}\label{def:formula-tau}
    Let $\overline{q}=(q, q', q'')$ be a sequence of propositional variables. We define the following formulas:
    \begin{align*}
        \tau(\overline{q}) \coloneq~ & q', \\
        \tau'(\overline{q}) \coloneq~ & \phi_{AB}[q'/\phi^*] \wedge \D q' \wedge \D q'' \wedge \B\neg(\D q \wedge \D q'' \wedge \B\neg q') \\
        & \wedge \A(q \vee q' \vee q'' \to \phi_{AB}[q'/\phi^*]), \\
        \tau''(\overline{q}) \coloneq~ & \phi_{AB}[q'/\phi^*] \wedge \neg q'' \wedge \D(q \wedge \bb\neg q'' \wedge \bd(q' \wedge \B\neg q'' \wedge \neg q)) \\
        &\wedge \D(q'' \wedge \bb\neg q' \wedge \B\neg q \wedge \bb\neg q) \wedge \B\neg q'.
    \end{align*}
\end{definition}

\begin{lemma}\label{lem:formula-tau}
    Let $V$ be a valuation on $\gf$. Then for all $i \in \omega$, 
    \begin{enumerate}
        \item The following are equivalent: 
        \begin{enumerate}
            \item[(i)] $\tup{V(q),V(q'),V(q'')} = \tup{\{b_i\}, \{b_{i+1}\}, \{a_i\}}$,
            \item[(ii)] $\tup{V(\tau),V(\tau'),V(\tau'')} = \tup{\{b_{i+1}\}, \{b_{i+2}\}, \{a_{i+1}\}}$;
        \end{enumerate}
        \item The following are equivalent: 
        \begin{enumerate}
            \item[(i)] $\tup{V(q),V(q'),V(q'')} = \tup{\{a_i\}, \{a_{i+1}\}, \{b_{i+1}\}}$,
            \item[(ii)] $\tup{V(\tau),V(\tau'),V(\tau'')} = \tup{\{a_{i+1}\}, \{a_{i+2}\}, \{b_{i+2}\}}$.
        \end{enumerate}
    \end{enumerate}
\end{lemma}
\begin{proof}
    We only prove (1); a proof of (2) can be obtained analogously. Suppose that (i) holds. Then $V(\tau) = V(q') = \set{b_{i+1}}$. It follows from Lemma~\ref{lem:S4t-formula-location-phiABq'} that $V(\phi_{AB}[q'/\phi^*]) = \set{a_i, b_i: i \in \omega}$. Using this, it is straightforward to verify that $V(\tau'') = \{a_{i+1}\}$ and $V(\tau') = \set{b_{i+2}}$.
    
    Conversely, suppose that (ii) holds. Then $V(q') = V(\tau) = \set{b_{i+1}}$. It follows from Lemma~\ref{lem:S4t-formula-location-phiABq'} that $V(\phi_{AB}[q'/\phi^*]) = \set{a_i, b_i: i \in \omega}$. Since $\gf,V,b_{i+2} \md \A(q \vee q' \vee q'' \to \phi_{AB}[q'/\phi^*])$, we have $V(q \vee q' \vee q'') \sub \AB$. Since $\gf,V,a_{i+1} \md \D(q \wedge \bb\neg q'' \wedge \bd(q' \wedge \B\neg q'' \wedge \neg q))$, there exists some $u\in R[a_{i+1}]$ such that $\gf,V,u \md q \wedge \bb\neg q'' \wedge \bd(q' \wedge \B\neg q'' \wedge \neg q)$, which entails  $\gf,V,b_{i+1} \md \B\neg q'' \wedge \neg q$. Since $\gf,V,a_{i+1} \md \D(q'' \wedge \bb\neg q' \wedge \B\neg q \wedge \bb\neg q) \wedge \neg q''$, there exists some $v \in R[a_{i+1}]$ such that $v \neq a_{i+1}$ and $\gf,V,v \md q'' \wedge \bb\neg q' \wedge \B\neg q \wedge \bb\neg q$. Then, $\gf,V,b_{i+1} \md \B\neg q''$ and $v \in V(q'') \sub \AB$. Thus, we have $v \not\in R[b_{i+1}]$, and so $v = a_{i}$. Since $\gf,V,a_{i} \md \B\neg q \wedge \bb\neg q$ and $\gf,V,b_{i+1} \md \neg q$, we have $V(q) = V(q) \cap AB = \set{b_{i}}$. Finally, by $\gf,V,b_{i+1} \md \B\neg q''$ and $\gf,V,b_{i} \md \bb\neg q''$, we have $V(q'') \cap AB  = \set{a_{i}}$.
\end{proof}

\begin{figure}[ht]
\[
    \begin{tikzpicture}[scale=1.2]

        \begin{scope}[shift={(0,0)}]
            \node (vd1) at (.75,3.4){$\vdots$};
            \node (vd2) at (0,3.4){$\vdots$};
            \node (vd3) at (1.5,3.4){$\vdots$};

            \node (al0) at (0,2.5)[label=left:$a_{i}$]{$\circ$};
            \node (al1) at (0,1.5)[label=left:$a_{i+1}$]{$\circ$};
            \node (al2) at (0,0.5)[label=left:$a_{i+2}$]{$\circ$};
            \node (bl0) at (1.5,2.5)[label=right:$b_{i}$]{$\circ$};
            \node (bl1) at (1.5,1.5)[label=right:$b_{i+1}$]{$\circ$};
            \node (bl2) at (1.50,0.5)[label=right:$b_{i+2}$]{$\circ$};

            \node at (-1.5,2.5) {$q''$};
            \node at (3,2.5) {$q$};
            \node at (3,1.5) {$q'$};

            \draw [->] (al1) -- (al0);
            \draw [->] (al2) -- (al1);
            \draw [->] (bl1) -- (bl0);
            \draw [->] (bl2) -- (bl1);
            \draw [->] (al1) -- (bl0);
            \draw [->] (al2) -- (bl1);
            \draw [->] (bl2) -- (al0);
            \node (vd4) at (0,0){$\vdots$};
            \node (vd5) at (.75,0){$\vdots$};
            \node (vd6) at (1.5,0){$\vdots$};

            \node at (4.5,1.5) {$\Longleftrightarrow$};
        \end{scope}

        \begin{scope}[shift={(7,0)}]
            \node (vd1) at (.75,3.4){$\vdots$};
            \node (vd2) at (0,3.4){$\vdots$};
            \node (vd3) at (1.5,3.4){$\vdots$};

            \node (al0) at (0,2.5)[label=left:$a_{i}$]{$\circ$};
            \node (al1) at (0,1.5)[label=left:$a_{i+1}$]{$\circ$};
            \node (al2) at (0,0.5)[label=left:$a_{i+2}$]{$\circ$};
            \node (bl0) at (1.5,2.5)[label=right:$b_{i}$]{$\circ$};
            \node (bl1) at (1.5,1.5)[label=right:$b_{i+1}$]{$\circ$};
            \node (bl2) at (1.50,0.5)[label=right:$b_{i+2}$]{$\circ$};

            \node at (-1.5,1.5) {$\tau''$};
            \node at (3,1.5) {$\tau$};
            \node at (3,.5) {$\tau'$};

            \draw [->] (al1) -- (al0);
            \draw [->] (al2) -- (al1);
            \draw [->] (bl1) -- (bl0);
            \draw [->] (bl2) -- (bl1);
            \draw [->] (al1) -- (bl0);
            \draw [->] (al2) -- (bl1);
            \draw [->] (bl2) -- (al0);
            \node (vd4) at (0,0){$\vdots$};
            \node (vd5) at (.75,0){$\vdots$};
            \node (vd6) at (1.5,0){$\vdots$};
        \end{scope}

        \begin{scope}[shift={(0,-5)}]
            \node (vd1) at (.75,3.4){$\vdots$};
            \node (vd2) at (0,3.4){$\vdots$};
            \node (vd3) at (1.5,3.4){$\vdots$};

            \node (al0) at (0,2.5)[label=left:$a_{i}$]{$\circ$};
            \node (al1) at (0,1.5)[label=left:$a_{i+1}$]{$\circ$};
            \node (al2) at (0,0.5)[label=left:$a_{i+2}$]{$\circ$};
            \node (bl0) at (1.5,2.5)[label=right:$b_{i}$]{$\circ$};
            \node (bl1) at (1.5,1.5)[label=right:$b_{i+1}$]{$\circ$};
            \node (bl2) at (1.50,0.5)[label=right:$b_{i+2}$]{$\circ$};

            \node at (-1.5,2.5) {$q$};
            \node at (-1.5,1.5) {$q'$};
            \node at (3,1.5) {$q''$};

            \draw [->] (al1) -- (al0);
            \draw [->] (al2) -- (al1);
            \draw [->] (bl1) -- (bl0);
            \draw [->] (bl2) -- (bl1);
            \draw [->] (al1) -- (bl0);
            \draw [->] (al2) -- (bl1);
            \draw [->] (bl2) -- (al0);
            \node (vd4) at (0,0){$\vdots$};
            \node (vd5) at (.75,0){$\vdots$};
            \node (vd6) at (1.5,0){$\vdots$};

            \node at (4.5,1.5) {$\Longleftrightarrow$};
        \end{scope}

        \begin{scope}[shift={(7,-5)}]
            \node (vd1) at (.75,3.4){$\vdots$};
            \node (vd2) at (0,3.4){$\vdots$};
            \node (vd3) at (1.5,3.4){$\vdots$};

            \node (al0) at (0,2.5)[label=left:$a_{i}$]{$\circ$};
            \node (al1) at (0,1.5)[label=left:$a_{i+1}$]{$\circ$};
            \node (al2) at (0,0.5)[label=left:$a_{i+2}$]{$\circ$};
            \node (bl0) at (1.5,2.5)[label=right:$b_{i}$]{$\circ$};
            \node (bl1) at (1.5,1.5)[label=right:$b_{i+1}$]{$\circ$};
            \node (bl2) at (1.50,0.5)[label=right:$b_{i+2}$]{$\circ$};

            \node at (-1.5,1.5) {$\tau$};
            \node at (-1.5,.5) {$\tau'$};
            \node at (3,.5) {$\tau''$};

            \draw [->] (al1) -- (al0);
            \draw [->] (al2) -- (al1);
            \draw [->] (bl1) -- (bl0);
            \draw [->] (bl2) -- (bl1);
            \draw [->] (al1) -- (bl0);
            \draw [->] (al2) -- (bl1);
            \draw [->] (bl2) -- (al0);
            \node (vd4) at (0,0){$\vdots$};
            \node (vd5) at (.75,0){$\vdots$};
            \node (vd6) at (1.5,0){$\vdots$};
        \end{scope}
    \end{tikzpicture}
\]
\caption{Illustration of Lemma~\ref{lem:formula-tau}}
\label{fig:formula-tau}
\end{figure}
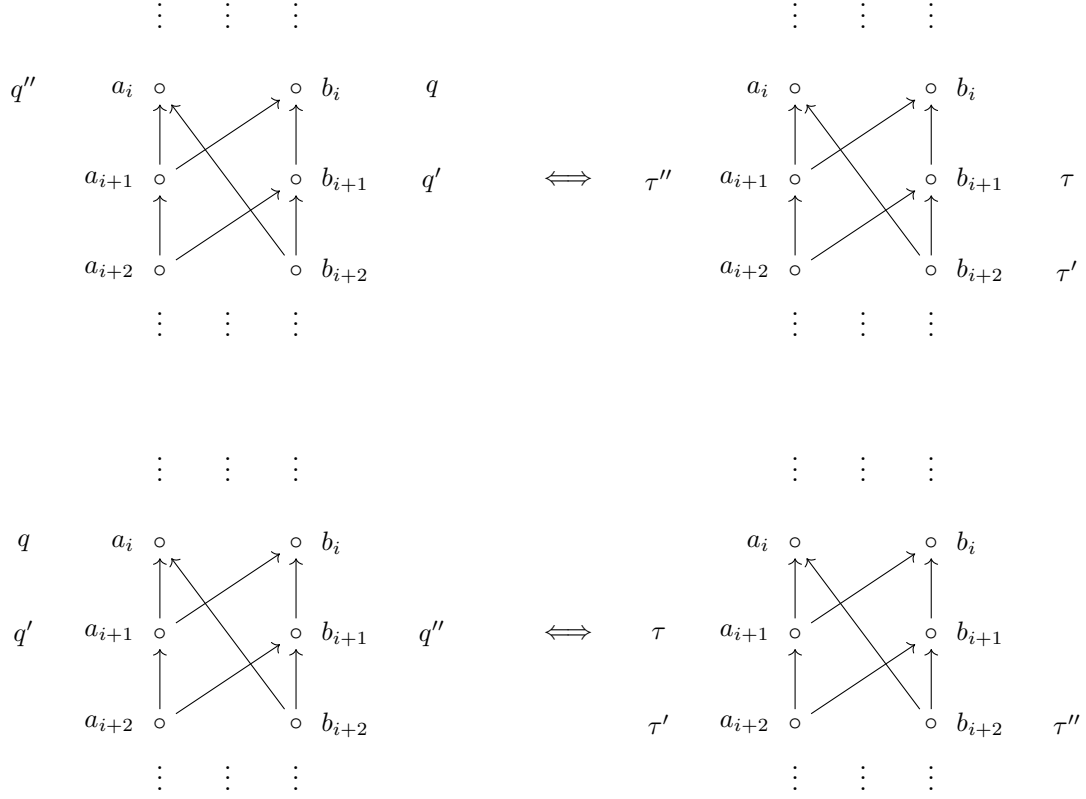

Using the formulas introduced in Definition~\ref{def:formula-tau}, we can define more subsets of $\gf$.

\begin{definition}\label{def:formula-aibi}
Recall $\phi_{b_0}$, $\phi_{b_1}$, and $\phi_{a_0}$ defined in Definition~\ref{def:S4t-char-formulas-1}. For each $i \geq 1$, we define 
\begin{align*}
    \phi_{b_{i+1}} \coloneq & \tau'(\phi_{b_{i-1}}, \phi_{b_{i}}, \phi_{a_{i-1}}), \\
    \phi_{a_i} \coloneq & \tau''(\phi_{b_{i-1}}, \phi_{b_i}, \phi_{a_{i-1}}).
\end{align*}
Moreover, we define the following formulas:
\begin{align*}
    \phi_{c'_{i}} &\coloneq \neg\phi_{AB} \wedge \D(\phi_{a_{i+1}} \wedge \phi_{AB}) \wedge \B\neg(\phi_{a_{i+2}} \wedge \phi_{AB}), \text{ for all } i \leq t_0, \\
    \phi_{c_0} &\coloneq \bd\phi_{c'_{0}} \wedge \neg\phi_{c'_{0}} \wedge \D\phi_{x_1} \wedge \neg\phi_{x_1}, \\
    \phi_{c_{i}} &\coloneq \D\phi_{c_0} \wedge \bd\phi_{c'_{i}} \wedge \bb\neg\phi_{c'_{i-1}} \wedge \neg\phi_{c'_i}, \text{ for all } 1\leq i\leq t_0, \\
    \phi_{R0} &\coloneq \bd\phi_{c_{t_0}} \wedge \bd\phi_{AB}, \\
    \phi_{R1} &\coloneq \B\D\phi_{R0} \wedge \neg\phi_{R0}, \\
    \phi_{R2} &\coloneq \D\phi_{R0} \wedge \neg\phi_{R0} \wedge \neg\phi_{R1} \wedge \B\neg\phi_{x_1} \wedge \neg\phi_{AB}, \\ 
    \phi_{R3} &\coloneq \bd\phi_{R2} \wedge \neg\phi_{R0} \wedge \neg\phi_{R2}.
\end{align*}
\end{definition}

\begin{lemma} \label{lem:S4t-formula-location-2}
    Let $V$ be a good valuation on $\gf$. Then, 
    \begin{enumerate}
        \item for all $u \in \set{a_i, b_i : i \in \omega}$, $V(\phi_{u}) = \set{u}$;
        \item for all $u \in \set{c_i, c'_i : i \leq t_0}$, $V(\phi_{u}) = \set{u}$;
        \item $V(\phi_{Ri}) = \set{z_i: z \in \Reach}$ for all $i \leq 3$.
    \end{enumerate}
\end{lemma}
\begin{proof}
    (1) follows from Lemmas~\ref{lem:S4t-formula-location-1} and \ref{lem:formula-tau} immediately. Using (1), it is straightforward to check (2) and (3).
\end{proof}

To simulate the transition of configurations of $\M$, we need a formula scheme that, under suitable substitutions, characterizes the points of the form $\tup{t,k,l}_0$ in $\gf$. In particular, the last two items in Lemma~\ref{lem:S4t-sigma-location} show that the second type of triangles in Lemma~\ref{lem:formula-tau} (cf. Figure~\ref{fig:formula-tau}) will not appear in our simulation of $\M$.

\begin{definition}
    For each state $t$ of $\M$ and formulas $\pi,\pi',\pi'',\kappa,\kappa',\kappa'' \in \FormT$, we define:
\begin{align*}
    \sigma(t, \pi, \pi', \pi'', \kappa, \kappa', \kappa'') &\coloneq \phi_{R0} \wedge \bd\phi_{c_{t+1}} \wedge \bb\neg\phi_{c_{t}} \\
    &\qquad \wedge \bd(\phi_{R1} \wedge \bd(\D\pi \wedge \D\pi'' \wedge \B\neg\pi') \wedge \bb\neg\pi' \wedge \bb\neg\pi'') \\
    &\qquad \wedge \bd(\phi_{R2} \wedge \bd\kappa' \wedge \bb\neg\kappa \wedge \bb\neg\kappa'').
\end{align*}
    We abbreviate $\sigma(t, \pi, \pi', \pi'', \kappa, \kappa', \kappa'')$ as $\sigma(t, \overline{\pi}, \overline{\kappa})$.
\end{definition}

\begin{lemma} \label{lem:S4t-sigma-location}
    Let $V$ be a good valuation on $\gf$. Then, the following hold:
    \begin{enumerate}
    \item Suppose
    \begin{align*}
        \tup{V(\pi),V(\pi'),V(\pi'')} &= \tup{\set{b_k},\set{b_{k+1}},\set{a_k}}, \text{ and } \\
        \tup{V(\kappa),V(\kappa'),V(\kappa'')} &= \tup{\set{b_l},\set{b_{l+1}},\set{a_l}},
    \end{align*}
    for some $k, l \in \omega$. Then,
    \begin{enumerate}
    \item[(i)] $\tup{t,k,l} \in \Reach$ implies $V(\sigma(t, \overline{\pi}, \overline{\kappa})) =\set{\tup{t,k,l}_0}$, and 
    \item[(ii)] $\tup{t,k,l} \not\in \Reach$ implies $V(\sigma(t, \overline{\pi}, \overline{\kappa})) = \ve$.
    \end{enumerate}
    \item Suppose $V(\pi) = \{a_{i}\}$, $V(\pi') = \{a_{i+1}\}$ and $V(\pi'') = \{b_{i+1}\}$ for some $i \in \omega$. Then $V(\sigma(t, \overline{\pi}, \overline{\kappa})) = \ve$.
    \item Suppose $V(\kappa) = \{a_{i}\}$, $V(\kappa') = \{a_{i+1}\}$ and $V(\kappa'') = \{b_{i+1}\}$ for some $i \in \omega$. Then $V(\sigma(t, \overline{\pi}, \overline{\kappa})) = \ve$.
    \end{enumerate}
\end{lemma}
\begin{proof}
    For (1), let $V$ be a good valuation on $\gf$ satisfying the assumptions. It suffices to show that $w \in V(\sigma(t, \overline{\pi}, \overline{\kappa}))$ iff $w = \tup{t,k,l}_0$. Take any $w \in V(\sigma(t, \overline{\pi}, \overline{\kappa}))$. Since $\gf,V,w \md \phi_{R0}$, by Lemma~\ref{lem:S4t-formula-location-2}(3), $w$ is of the form $\tup{t',k',l'}_0$. Since $\gf,V,w \md \bd\phi_{c_{t+1}} \wedge \bb\neg\phi_{c_{t}}$, by Lemma~\ref{lem:S4t-formula-location-2}(2), we see that $c_{t+1} R w$ and $w \not\in R[c_{t}]$, which entails $t' = t$. It follows from the assumption that $V(\D\pi \wedge \D\pi'' \wedge \B\neg\pi') = \set{a_{k+1}}$. Since $\gf,V,w \md \bd(\phi_{R1} \wedge \bd(\D\pi \wedge \D\pi'' \wedge \B\neg\pi')\wedge \bb\neg\pi' \wedge \bb\neg\pi'')$, we see that $\gf,V,\tup{t',k',l'}_1 \md \bd(\D\pi \wedge \D\pi'' \wedge \B\neg\pi') \wedge \bb\neg\pi''$, which entails that $\tup{t',k',l'}_1 \in R[a_{k+1}] \setminus R[a_{k}]$, and so $k' = k$. By $\gf,V,w \md \bd(\phi_{R2} \wedge \bd\kappa' \wedge \bb\neg\kappa \wedge \bb\neg\kappa'')$, we see that $\gf,V,\tup{t',k',l'}_2 \md \bd\kappa' \wedge \bb\neg\kappa$, which entails that $\tup{t',k',l'}_2 \in R[b_{l+1}] \setminus R[b_{l}]$, and so $l' = l$. Thus, $w = \tup{t,k,l}_0$. 
    
    It remains to check that $\gf,V,\tup{t,k,l}_0 \md \sigma(t, \overline{\pi}, \overline{\kappa})$, assuming $\tup{t,k,l} \in \Reach$. It is standard to verify the following satisfactions: 
    \begin{align*}
        \gf,V,\tup{t,k,l}_0 &\md \phi_{R0} \wedge \bd\phi_{c_{t+1}} \wedge \bb\neg\phi_{c_{t}} \\
        \gf,V,\tup{t,k,l}_1 &\md \phi_{R1} \wedge \bd(\D\pi \wedge \D\pi'' \wedge \B\neg\pi') \wedge \bb\neg\pi' \wedge \bb\neg\pi'', \\
        \gf,V,\tup{t,k,l}_2 &\md \phi_{R2} \wedge \bd\kappa' \wedge \bb\neg\kappa \wedge \bb\neg\kappa''.
    \end{align*}
    Thus, $\gf, V, \tup{t,k,l}_0 \md \sigma(t, \overline{\pi}, \overline{\kappa})$. Hence, (1) holds.
    
    For (2), let $V$ be a good valuation such that $V(\pi) = \{a_{i}\}$, $V(\pi') = \{a_{i+1}\}$, and $V(\pi'') = \{b_{i+1}\}$ for some $i \in \omega$. Then we see that $V(\D\pi \wedge \D\pi'' \wedge \B\neg\pi') = \set{b_{i+2}}$. By the construction of $\gf$, for any $z \in \Reach$ such that $z_1 \in R[b_{i+2}]$, we have $z_1 \in R[a_i] \sub R[a_{i+1}]$. By Lemma~\ref{lem:S4t-formula-location-2}, $\gf,V \md \phi_{R1} \wedge \bd(\D\pi \wedge \D\pi'' \wedge \B\neg\pi') \to \bd\pi'$, which entails that $V(\sigma(t, \overline{\pi}, \overline{\kappa})) = \ve$.

    For (3), let $V$ be a good valuation such that $V(\kappa) = \{a_{i}\}$, $V(\kappa') = \{a_{i+1}\}$, and $V(\kappa'') = \{b_{i+1}\}$ for some $i \in \omega$. Then, it is clear that for all $z \in \Reach$, if $z_2 \in R[a_{i+1}]$, then $z_2 \in R[b_{i+1}]$. By Lemma~\ref{lem:S4t-formula-location-2}, $\gf,V \md \phi_{R2} \wedge \bd\kappa' \to \bd\kappa''$, and so $V(\sigma(t, \overline{\pi}, \overline{\kappa})) = \ve$. 
\end{proof}

In Lemma~\ref{lem:S4t-sigma-location}, one may have noticed that $\sigma(t, \overline{\pi}, \overline{\kappa})$ has the intended meaning only if the substituted formulas satisfy a very specific condition, namely, the assumption in (1). Thus, before introducing the axioms that simulate the instructions of $\M$, we need the following formulas to characterize the valuations that meet this condition. Intuitively, the formula $\psi_\tri(\overline{q})$ indicates a triangle and the formulas $\psi(\overline{q})$, $\psi'(\overline{q})$, and $\psi''(\overline{q})$ point to the three vertices of the triangle. The $(-)^+$ operation moves the triangle one step below.

\begin{definition}
    Let $\overline{q} = (q,q',q'')$ be a sequence of propositional variables. We define the following formulas: 
    \begin{align*}
    \psi_\tri(\overline{q}) \coloneq & \A(q \vee q' \vee q'' \to \phi_{AB})\\
    & \wedge \E(\phi_{AB} \wedge q \wedge \B\neg q' \wedge \bd q' \wedge \B\neg q'' \wedge \bb\neg q'') \\
    &\wedge \E(\phi_{AB} \wedge q' \wedge \D q \wedge \bb\neg q \wedge \B\neg q'' \wedge \bb\neg q'')\\
    & \wedge \E(\phi_{AB} \wedge q'' \wedge \B\neg q \wedge \bb\neg q \wedge \B\neg q' \wedge \bb\neg q'), \\
    \psi(\overline{q}) \coloneq & \psi_\tri(\overline{q}) \land q,\\ 
    \psi'(\overline{q}) \coloneq & \psi_\tri(\overline{q}) \land q',\\ 
    \psi''(\overline{q}) \coloneq & \psi_\tri(\overline{q}) \land q''.
    \end{align*} 
    Moreover, for each formula $\phi(\overline{q})$, let 
    \[\phi^+(\overline{q}) \coloneq \phi(\tau(\overline{q}), \tau'(\overline{q}), \tau''(\overline{q})).\]
\end{definition}

Note that $\phi_{AB}$ has no occurrence of $q,q',q''$, so it remains the same under $(-)^+$.

\begin{lemma}\label{lem:S4t-psi-location}
    Let $V$ be a good valuation on $\gf$. Then, the following are equivalent:
    \begin{enumerate}
        \item $V(\psi_\tri(\overline{q})) \neq \ve$;
        \item $V(\psi_\tri(\overline{q})) = W$;
        \item there exists $i \in \omega$ such that one of the following holds:
    \begin{enumerate}
        \item[(a)] $V(q) = \set{b_i}$, $V(q') = \set{b_{i+1}}$ and $V(q'') = \set{a_{i}}$, or
        \item[(b)] $V(q) = \set{a_i}$, $V(q') = \set{a_{i+1}}$ and $V(q'') = \set{b_{i+1}}$.
    \end{enumerate}
    \end{enumerate}
\end{lemma}
\begin{proof}
    The equivalence between (1) and (2) follows from the fact that $\gf \md \axiom{br}_{7}$. The implication from (3) to (1) is straightforward. We now prove that (1) implies (3). Suppose $V(\psi_\tri(\overline{q})) \neq \ve$. Then $V(q \vee q' \vee q'') \sub \AB$. By Lemma~\ref{lem:S4t-formula-location-1}, there exist $u,u',u'' \in \AB$ such that 
    \begin{align*}
        \gf,V,u\:\: &\md q \wedge \B\neg q'' \wedge \bb\neg q'' \wedge \B\neg q', \\
        \gf,V,u'\: &\md q' \wedge \B\neg q'' \wedge \bb\neg q'', \text{ and } \\
        \gf,V,u'' &\md q''.
    \end{align*}
    Then, we see that $\set{u,u'} \cap (R[u''] \cup \inver{R}[u'']) = \ve$ and $u' \not\in R[u]$. By the construction of $\gf$, either $u'' = a_i$ or $u'' = b_{i+1}$ for some $i \in \omega$. Suppose $u'' = a_i$. Then we see that $u = b_i$ and $u' = b_{i+1}$. Since $\gf,V,a_i \md \B\neg q \wedge \bb\neg q$ and $\gf,V,b_{i+1} \md \bb\neg q$, we have $V(q) = \set{b_i}$. It is also straightforward to check that $V(q') = \set{b_{i+1}}$ and $V(q'') = \set{a_{i}}$. Thus, (a) holds. Similarly, if $u'' = b_{i+1}$, then (b) holds. Hence, (1) implies (3).
\end{proof}

In other words, Lemma~\ref{lem:S4t-psi-location} says that $\psi_\tri(\overline{q})$ is satisfied in $(\gf,V)$ if and only if $V(\overline{q})$ forms a triangle depicted in Figure~\ref{fig:psiS}.

\begin{figure}[ht]
\[
    \begin{tikzpicture}[scale=1.2]

        \begin{scope}[shift={(0,0)}]
            \node (vd1) at (.75,3.4){$\vdots$};
            \node (vd2) at (0,3.4){$\vdots$};
            \node (vd3) at (1.5,3.4){$\vdots$};

            \node (al0) at (0,2.5)[label=left:$a_{i}$]{$\circ$};
            \node (al1) at (0,1.5)[label=left:$a_{i+1}$]{$\circ$};
            \node (al2) at (0,0.5)[label=left:$a_{i+2}$]{$\circ$};
            \node (bl0) at (1.5,2.5)[label=right:$b_{i}$]{$\circ$};
            \node (bl1) at (1.5,1.5)[label=right:$b_{i+1}$]{$\circ$};
            \node (bl2) at (1.50,0.5)[label=right:$b_{i+2}$]{$\circ$};

            \node at (-1.5,2.5) {$q''$};
            \node at (3,2.5) {$q$};
            \node at (3,1.5) {$q'$};

            \draw [->] (al1) -- (al0);
            \draw [->] (al2) -- (al1);
            \draw [->] (bl1) -- (bl0);
            \draw [->] (bl2) -- (bl1);
            \draw [->] (al1) -- (bl0);
            \draw [->] (al2) -- (bl1);
            \draw [->] (bl2) -- (al0);
            \node (vd4) at (0,0){$\vdots$};
            \node (vd5) at (.75,0){$\vdots$};
            \node (vd6) at (1.5,0){$\vdots$};
        \end{scope}

        \begin{scope}[shift={(7,0)}]
            \node (vd1) at (.75,3.4){$\vdots$};
            \node (vd2) at (0,3.4){$\vdots$};
            \node (vd3) at (1.5,3.4){$\vdots$};

            \node (al0) at (0,2.5)[label=left:$a_{i}$]{$\circ$};
            \node (al1) at (0,1.5)[label=left:$a_{i+1}$]{$\circ$};
            \node (al2) at (0,0.5)[label=left:$a_{i+2}$]{$\circ$};
            \node (bl0) at (1.5,2.5)[label=right:$b_{i}$]{$\circ$};
            \node (bl1) at (1.5,1.5)[label=right:$b_{i+1}$]{$\circ$};
            \node (bl2) at (1.50,0.5)[label=right:$b_{i+2}$]{$\circ$};

            \node at (-1.5,2.5) {$q$};
            \node at (-1.5,1.5) {$q'$};
            \node at (3,1.5) {$q''$};

            \draw [->] (al1) -- (al0);
            \draw [->] (al2) -- (al1);
            \draw [->] (bl1) -- (bl0);
            \draw [->] (bl2) -- (bl1);
            \draw [->] (al1) -- (bl0);
            \draw [->] (al2) -- (bl1);
            \draw [->] (bl2) -- (al0);
            \node (vd4) at (0,0){$\vdots$};
            \node (vd5) at (.75,0){$\vdots$};
            \node (vd6) at (1.5,0){$\vdots$};
        \end{scope}
    \end{tikzpicture}
\]
\caption{Possible cases where $\psi_\tri(\overline{q})$ is satisfied}
\label{fig:psiS}
\end{figure}
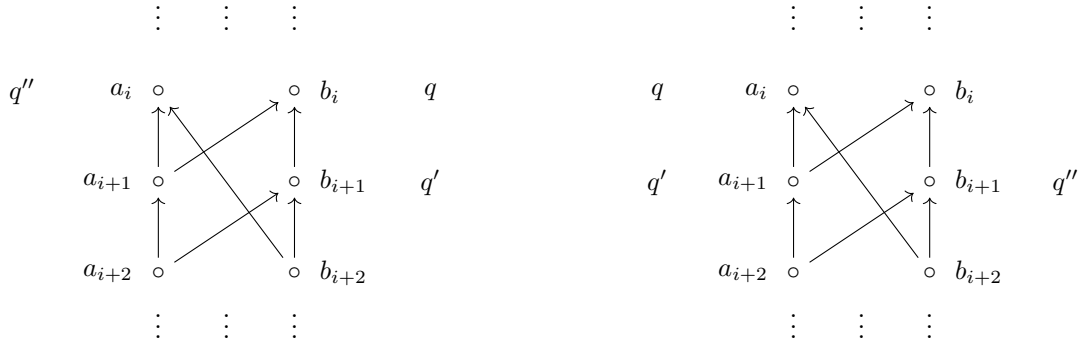

\begin{lemma} \label{lem:S4t-psi+-location}
    Let $V$ be a good valuation. 
    Suppose that at least one of $V(\psi_\tri(\overline{q}))$ and $V(\psi_\tri^+(\overline{q}))$ is non-empty. Then, there exists some $i \in \omega$ such that one of the following holds:
    \begin{enumerate}
        \item $V(\psi) = \{b_i\}$, $V(\psi') = V(\psi^+) = \{b_{i+1}\}$, $V(\psi'') = \{a_i\}$, $V(\psi'^+) = \{b_{i+2}\}$, and $V(\psi''^+) = \{a_{i+1}\}$;
        \item $V(\psi) = \{a_i\}$, $V(\psi') = V(\psi^+) = \{a_{i+1}\}$, $V(\psi'') = \{b_{i+1}\}$, $V(\psi'^+) = \{a_{i+2}\}$, and $V(\psi''^+) = \{b_{i+2}\}$.
    \end{enumerate}
\end{lemma}
\begin{proof}
    First, suppose $V(\psi_\tri(\overline{q})) \neq \ve$. 
    By Lemma~\ref{lem:S4t-psi-location}, there are two cases to consider. The first case is that $V(q) = \set{b_i}$, $V(q') = \set{b_{i+1}}$, and $V(q'') = \set{a_{i}}$ for some $i \in \omega$. By Lemma~\ref{lem:formula-tau}, we have $V(\tau) = \set{b_{i+1}}$, $V(\tau') = \set{b_{i+2}}$, and $V(\tau'') = \set{a_{i+1}}$. By Lemma~\ref{lem:S4t-psi-location}, we have $V(\psi_\tri^+(\overline{q})) = V(\psi_\tri(\overline{q})) = W$, which implies (1). The second case is that $V(q) = \set{a_i}$, $V(q') = \set{a_{i+1}}$, and $V(q'') = \set{b_{i+1}}$ for some $i \in \omega$. A similar argument with Lemmas~\ref{lem:formula-tau} and \ref{lem:S4t-psi-location} shows that (2) holds.

    Next, suppose $V(\psi_\tri^+(\overline{q})) = V(\psi_\tri(\tau,\tau',\tau'')) \neq \ve$. 
    By the proof of Lemma~\ref{lem:S4t-psi-location}, we have similar case distinctions as in the previous paragraph. The first case is that $V(\tau) = \set{b_j}$, $V(\tau') = \set{b_{j+1}}$, and $V(\tau'') = \set{a_{j}}$ for some $j \in \omega$. By $\gf,V,a_{j} \md \tau''$, we have $\gf,V,a_{j} \md \D q'' \wedge \neg q''$, which entails $j \neq 0$. Let $i = j - 1$. Then $V(\tau) = \set{b_{i+1}}$, $V(\tau') = \set{b_{i+2}}$, and $V(\tau'') = \set{a_{i+1}}$. By Lemma~\ref{lem:formula-tau}, we have $V(q) = \set{b_{i}}$, $V(q') = \set{b_{i+1}}$ and $V(q'') = \set{a_{i}}$. Thus, (1) holds.
    
    The second case is that $V(\tau) = \set{a_j}$, $V(\tau') = \set{a_{j+1}}$, and $V(\tau'') = \set{b_{j+1}}$ for some $j \in \omega$. By $\gf,V,b_{j+1} \md \tau''$, we have $\gf,V,b_{j+1} \md \D(q \wedge \bd q')$. Since $V(q') = V(\tau) = \set{a_j}$ and $a_0 \not\in \inver{R}[R[b_1]]$, we see that $j \neq 0$. Let $i = j-1$. Then, $V(\tau) = \set{a_{i+1}}$, $V(\tau') = \set{a_{i+2}}$, and $V(\tau'') = \set{b_{i+2}}$. By Lemma~\ref{lem:formula-tau}, we have $V(q) = \set{a_{i}}$, $V(q') = \set{a_{i+1}}$, and $V(q'') = \set{b_{i+1}}$. By Lemma~\ref{lem:S4t-psi-location}, we have $V(\psi_\tri^+(\overline{q})) = V(\psi_\tri(\overline{q})) = W$, which implies (2).
\end{proof}

For $i \in \omega$, let $[i/\overline{q}]$ be the substitution $[\phi_{b_i}/q, \phi_{b_{i+1}}/q', \phi_{a_i}/q'']$. We also write $\psi(i)$ for $\psi[i/\overline{q}]$. Moreover, as usual, we abbreviate $(\psi(\overline{q}), \psi'(\overline{q}), \psi''(\overline{q}))$ as $\overline{\psi}(\overline{q})$. These notations apply to other formulas as well. For example, under our notation, $\sigma(t, \overline{\psi}(\overline{q}), \overline{\psi}(0))$ is the substitution result of
\[\sigma(t, \psi(\overline{q}), \psi'(\overline{q}), \psi''(\overline{q}), \psi(\overline{r}), \psi'(\overline{r}), \psi''(\overline{r})) \; [\phi_{b_0}/r, \phi_{b_1}/r', \phi_{a_0}/r''].\]

Now with all these preparations, we are ready to introduce the axioms simulating the instructions in $\M$. Intuitively, each $AxI$ says that if the current configuration of $\M$ is represented by the antecedent, then the consequent represents the next configuration of $\M$ after executing~$I$.

\begin{definition}
With each instruction $I$ in $\M$ we associate a formula $AxI$ by taking:
\begin{itemize}
    \item If $I=t\to\tup{t',1,0}$, then
    \[
    AxI \coloneq \phi^* \wedge \E\sigma(t, \overline{\psi}(\overline{q}), \overline{\psi}(\overline{r})) \to \phi^* \wedge \E\sigma(t', \overline{\psi^+}(\overline{q}), \overline{\psi}(\overline{r})).
    \]
    \item If $I=t\to\tup{t',0,1}$, then
    \[
    AxI \coloneq \phi^* \wedge \E\sigma(t, \overline{\psi}(\overline{q}), \overline{\psi}(\overline{r})) \to \phi^* \wedge \E\sigma(t', \overline{\psi}(\overline{q}), \overline{\psi^+}(\overline{r})).
    \]
    \item If $I=t\to\tup{t',-1,0} (\tup{t'',0,0})$, then
    \begin{align*}
        AxI \coloneq & (\phi^* \wedge \E\sigma(t, \overline{\psi^+}(\overline{q}), \overline{\psi}(\overline{r})) \to \phi^* \wedge \E\sigma(t', \overline{\psi}(\overline{q}), \overline{\psi}(\overline{r}))) \\
        & \wedge (\phi^* \wedge \E\sigma(t,\overline{\psi}(0), \overline{\psi}(\overline{r})) \to \phi^* \wedge \E\sigma(t'',\overline{\psi}(0), \overline{\psi}(\overline{r}))).
    \end{align*}

    \item If $I=t\to\tup{t',0,-1} (\tup{t'',0,0})$, then
    \begin{align*}
        AxI \coloneq & (\phi^* \wedge \E\sigma(t, \overline{\psi}(\overline{q}), \overline{\psi^+}(\overline{r})) \to \phi^* \wedge \E\sigma(t', \overline{\psi}(\overline{q}), \overline{\psi}(\overline{r}))) \\
        & \wedge (\phi^* \wedge \E\sigma(t,\overline{\psi}(\overline{q}),\overline{\psi}(0)) \to \phi^* \wedge \E\sigma(t'',\overline{\psi}(\overline{q}),\overline{\psi}(0))).
    \end{align*}
\end{itemize}
Finally, let $AxM\coloneq\bigwedge_{I\in \M}AxI$.
\end{definition}

Note that by definition $\M$ has only finitely many instructions, so $AxM$ is a well-defined formula. 
In what follows, we show that $AxM$ simulates $\M$ semantically, in the sense of Lemma~\ref{lem:S4t-gf-AxM}, and syntactically, in the sense of Lemma~\ref{lem:S4t-syntax-main}.

\begin{lemma} \label{lem:S4t-gf-AxM}
    $\gf \md AxM$.
\end{lemma}

\begin{proof}
    It suffices to show that $\gf \md AxI$ for each instruction $I \in \M$. We show the cases where $I=t\to\tup{t',1,0}$ and $I=t\to\tup{t',-1,0} (\tup{t'',0,0})$; the other two cases are similar.

    For the case $I=t\to\tup{t',1,0}$, let $V$ be a valuation on $\gf$ and $w \in W$. Suppose that
    \[\gf, V, w \md \phi^* \wedge \E\sigma(t, \overline{\psi}(\overline{q}), \overline{\psi}(\overline{r})).\]
    Then, $V(\phi^*) \neq \ve$ and $V(\sigma(t, \overline{\psi}(\overline{q}), \overline{\psi}(\overline{r}))) \neq \ve$. By the construction of $\sigma$, we have $V(\overline{\psi}(\overline{q})) \neq \ve$ and so $V(\psi_\tri(\overline{q})) \neq \ve$. By Lemma~\ref{lem:S4t-psi+-location}, either (i) $V(\psi'(\overline{q})) = \{b_{k+1}\}$ and $V(\psi''(\overline{q})) = \{a_k\}$ or (ii) $V(\psi'(\overline{q})) = \{a_{k+1}\}$ and $V(\psi''(\overline{q})) = \{b_{k+1}\}$ for some $k \in \omega$. It follows from Lemma~\ref{lem:S4t-sigma-location}(2) that (ii) is impossible. Thus, by Lemma~\ref{lem:S4t-psi+-location} again, we have $V(\psi(\overline{q})) = \{b_k\}$, $V(\psi'(\overline{q})) = V(\psi^+(\overline{q})) = \{b_{k+1}\}$, $V(\psi''(\overline{q})) = \{a_k\}$, $V(\psi'^+(\overline{q})) = \{b_{k+2}\}$ and $V(\psi''^+(\overline{q})) = \{a_{k+1}\}$. Similarly, by Lemmas~\ref{lem:S4t-psi+-location} and \ref{lem:S4t-sigma-location}(3), we see that $V(\psi(\overline{r})) = \{b_l\}$, $V(\psi'(\overline{r})) = V(\psi^+(\overline{r})) = \{b_{l+1}\}$, $V(\psi''(\overline{r})) = \{a_l\}$, $V(\psi'^+(\overline{r})) = \{b_{l+2}\}$ and $V(\psi''^+(\overline{r})) = \{a_{l+1}\}$ for some $l \in \omega$. 
    
    Thus, it follows from Lemma~\ref{lem:S4t-sigma-location}(1) that 
    \[V(\sigma(t, \overline{\psi}(\overline{q}), \overline{\psi}(\overline{r}))) = \set{\tup{t,k,l}_0}.\]
    This implies that $\tup{t,k,l} \in \Reach$, which yields $\tup{t',k+1,l} \in \Reach$ by applying $I=t\to\tup{t',1,0}$. Again by Lemma~\ref{lem:S4t-sigma-location}(1), we have 
    \[V(\sigma(t', \overline{\psi^+}(\overline{q}), \overline{\psi}(\overline{r}))) = \set{\tup{t',k+1,l}_0} \neq \ve.\]
    Thus, 
    \[\gf, V, w \md \phi^* \wedge \E\sigma(t', \overline{\psi^+}(\overline{q}), \overline{\psi}(\overline{r})),\]
    which entails that $\gf, V, w \md AxI$.
    
    For the case $I=t\to\tup{t',-1,0} (\tup{t'',0,0})$, let $V$ be a valuation on $\gf$ and $w \in W$. Then we have two subcases.
    
    First, suppose that
    \[\gf, V, w \md \phi^* \wedge \E\sigma(t, \overline{\psi^+}(\overline{q}), \overline{\psi}(\overline{r})).\]
    Then, $V(\phi^*) \neq \ve$ and $V(\sigma(t, \overline{\psi^+}(\overline{q}), \overline{\psi}(\overline{r}))) \neq \ve$. As in the previous case, by a similar argument using Lemmas~\ref{lem:S4t-psi+-location} and \ref{lem:S4t-sigma-location}, we have $V(\psi(\overline{q})) = \{b_k\}$, $V(\psi'(\overline{q})) = V(\psi^+(\overline{q})) = \{b_{k+1}\}$, $V(\psi''(\overline{q})) = \{a_k\}$, $V(\psi'^+(\overline{q})) = \{b_{k+2}\}$ and $V(\psi''^+(\overline{q})) = \{a_{k+1}\}$ for some $k \in \omega$, and $V(\psi(\overline{r})) = \{b_l\}$, $V(\psi'(\overline{r})) = V(\psi^+(\overline{r})) = \{b_{l+1}\}$, $V(\psi''(\overline{r})) = \{a_l\}$, $V(\psi'^+(\overline{r})) = \{b_{l+2}\}$ and $V(\psi''^+(\overline{r})) = \{a_{l+1}\}$ for some $l \in \omega$. 
    
    Thus, it follows from Lemma~\ref{lem:S4t-sigma-location}(1) that 
    \[V(\sigma(t, \overline{\psi^+}(\overline{q}), \overline{\psi}(\overline{r}))) = \set{\tup{t,k+1,l}_0}.\]
    This implies that $\tup{t,k+1,l} \in \Reach$, which yields $\tup{t',k,l} \in \Reach$ by applying $I$. Again by Lemma~\ref{lem:S4t-sigma-location}(1), we have 
    \[V(\sigma(t', \overline{\psi}(\overline{q}), \overline{\psi}(\overline{r}))) = \set{\tup{t',k,l}_0} \neq \ve.\]
    Thus, 
    \[\gf, V, w \md \phi^* \wedge \E\sigma(t', \overline{\psi}(\overline{q}), \overline{\psi}(\overline{r})).\]
    
    Next, suppose that 
    \[\gf, V, w \md \phi^* \wedge \E\sigma(t, \overline{\psi}(0), \overline{\psi}(\overline{r})).\]
    Then, $V(\phi^*) \neq \ve$ and $V(\sigma(t, \overline{\psi}(0), \overline{\psi}(\overline{r}))) \neq \ve$. As in the previous case, by a similar argument using Lemmas~\ref{lem:S4t-psi+-location} and \ref{lem:S4t-sigma-location}, we have $V(\psi(\overline{r})) = \{b_l\}$, $V(\psi'(\overline{r})) = V(\psi^+(\overline{r})) = \{b_{l+1}\}$, $V(\psi''(\overline{r})) = \{a_l\}$, $V(\psi'^+(\overline{r})) = \{b_{l+2}\}$ and $V(\psi''^+(\overline{r})) = \{a_{l+1}\}$ for some $l \in \omega$. 
    
    Thus, it follows from Lemma~\ref{lem:S4t-sigma-location}(1) that 
    \[V(\sigma(t, \overline{\psi}(0), \overline{\psi}(\overline{r}))) = \set{\tup{t,0,l}_0}.\]
    This implies that $\tup{t,0,l} \in \Reach$. By applying instruction $I$, we see that $\tup{t'',0,l} \in \Reach$. Again by Lemma~\ref{lem:S4t-sigma-location}(1), we have 
    \[V(\sigma(t'', \overline{\psi}(0), \overline{\psi}(\overline{r}))) = \set{\tup{t'',0,l}_0} \neq \ve.\]
    Thus, 
    \[
    \gf, V, w \md \phi^* \wedge \E\sigma(t'', \overline{\psi}(0), \overline{\psi}(\overline{r})),
    \]
    which entails that $\gf, V, w \md AxI$. Since $w$ and $V$ are arbitrarily chosen, $\gf \md AxI$. The other two cases follow analogously, and we conclude $\gf \md AxI$.
\end{proof}

Now we turn to syntactic analysis of $AxM$. Recall that $\psi(i)$ denotes the substitution result of $\psi(q, q', q'')[\phi_{b_i}/q, \phi_{b_{i+1}}/q', \phi_{a_i}/q'']$ and similarly for other formulas.

\begin{lemma} \label{lem:S4t-syntax-psi+}
    For any $i \in \omega$, 
    \begin{align*}
        \psi^+(i) &= \psi(i+1), \\
        \psi'^+(i) &= \psi'(i+1), \\
        \psi''^+(i) &= \psi''(i+1).
    \end{align*}
\end{lemma}
\begin{proof}
    We only show the first case. Note that $\phi_{AB}$ has no occurrence of $q, q', q''$. By Definitions~\ref{def:formula-tau}~and~\ref{def:formula-aibi}, we have
    \begin{align*}
        \tau(i) &= \phi_{b_{i+1}},\\
        \tau'(i) &= \phi_{b_{i+2}}, \\
        \tau''(i) &= \phi_{a_{i+1}}.
    \end{align*}
    So,
    \begin{align*}
        \psi^+(i) &= (\psi_\tri(\tau(\overline{q}), \tau'(\overline{q}), \tau''(\overline{q})) \land \tau(\overline{q}))[i/\overline{q}] \\
        &= \psi_\tri[\tau(i), \tau'(i), \tau''(i)] \land \tau(i) \\
        &= \psi_\tri(\phi_{b_{i+1}}, \phi_{b_{i+2}}, \phi_{a_{i+1}}) \land \phi_{b_{i+1}} \\
        &= \psi(\phi_{b_{i+1}}, \phi_{b_{i+2}}, \phi_{a_{i+1}}) \\
        &= \psi(i+1).
    \end{align*}
    The other two cases follow analogously. 
\end{proof}

By Lemma~\ref{lem:S4t-syntax-psi+}, we can prove the following lemma.

\begin{lemma} \label{lem:S4t-syntax-main}
    For any configuration $\tup{t,k,l} \in \Reach$, 
    \begin{align*}
        &\phi^* \wedge \E\sigma(s, \overline{\psi}(n), \overline{\psi}(m)) \to \phi^* \wedge \E\sigma(t, \overline{\psi}(k), \overline{\psi}(l)) \in \TL{Grz}\oplus AxM.
    \end{align*}
\end{lemma}
\begin{proof}
    We prove by induction on the length of the computation of $\M$. The base case, where $\tup{t,k,l} = \tup{s,n,m}$, is clear since the corresponding formula is a tautology. Consider the computation of the form $\tup{s,n,m} \rightsquigarrow \tup{t,k,l} \to \tup{\Tilde{t},\Tilde{k},\Tilde{l}}$, where $I$ is the last instruction applied. By induction hypothesis, we have
    \begin{align*}
        &\phi^* \wedge \E\sigma(s, \overline{\psi}(n), \overline{\psi}(m)) \to \phi^* \wedge \E\sigma(t, \overline{\psi}(k), \overline{\psi}(l))  \in \TL{Grz}\oplus AxM.
    \end{align*}
    So, it suffices to show 
    \begin{align*}
        &\phi^* \wedge \E\sigma(t, \overline{\psi}(k), \overline{\psi}(l)) \to \phi^* \wedge \E\sigma(\Tilde{t}, \overline{\psi}(\Tilde{k}), \overline{\psi}(\Tilde{l}))  \in \TL{Grz}\oplus AxM.
    \end{align*}
    We distinguish cases according to the form of $I$.
    \begin{itemize}
        \item $I=t\to\tup{t',1,0}$. Then $\tup{\Tilde{t},\Tilde{k},\Tilde{l}} = \tup{t', k+1, l}$. From $AxI$ we have
        \begin{align*}
            &\phi^* \wedge \E\sigma(t, \overline{\psi}(\overline{q}), \overline{\psi}(\overline{r})) \to \phi^* \wedge \E\sigma(t', \overline{\psi^+}(\overline{q}), \overline{\psi}(\overline{r})) \in \TL{Grz}\oplus AxM.
        \end{align*}
        Applying the substitution $[k/\overline{q}, l/\overline{r}]$, with Lemma~\ref{lem:S4t-syntax-psi+}, we obtain
        \begin{align*}
            &\phi^* \wedge \E\sigma(t, \overline{\psi}(k), \overline{\psi}(l)) \to \phi^* \wedge \E\sigma(t', \overline{\psi}(k+1), \overline{\psi}(l))  \in \TL{Grz}\oplus AxM.
        \end{align*}
        
        \item $I=t\to\tup{t',0,1}$. This case is similar to the previous one.
        
        \item $I=t\to\tup{t',-1,0} (\tup{t'',0,0})$. We further distinguish cases depending on whether $k=0$.
        \begin{itemize}
            \item[\labelitemiii] $k \geq 1$. Then $\tup{\Tilde{t},\Tilde{k},\Tilde{l}} = \tup{t', k-1, l}$. From $AxI$ we have 
            \begin{align*}
                &\phi^* \wedge \E\sigma(t, \overline{\psi^+}(\overline{q}), \overline{\psi}(\overline{r})) \to \phi^* \wedge \E\sigma(t', \overline{\psi}(\overline{q}), \overline{\psi}(\overline{r})) \in \TL{Grz}\oplus AxM.
            \end{align*}
            Applying the substitution $[k-1/\overline{q}, l/\overline{r}]$, with Lemma~\ref{lem:S4t-syntax-psi+}, we obtain 
            \begin{align*}
                &\phi^* \wedge \E\sigma(t, \overline{\psi}(k), \overline{\psi}(l)) \to \phi^* \wedge \E\sigma(t', \overline{\psi}(k-1), \overline{\psi}(l))  \in \TL{Grz}\oplus AxM.
            \end{align*}
            \item[\labelitemiii] $k = 0$. Then $\tup{\Tilde{t},\Tilde{k},\Tilde{l}} = \tup{t'', 0, l}$. From $AxI$ we have 
            \begin{align*}
                &\phi^* \wedge \E\sigma(t,\overline{\psi}(0), \overline{\psi}(\overline{r})) \to \phi^* \wedge \E\sigma(t'',\overline{\psi}(0), \overline{\psi}(\overline{r})) \in \TL{Grz}\oplus AxM.
            \end{align*}
            Applying the substitution $[l/\overline{r}]$, we obtain 
            \begin{align*}
                &\phi^* \wedge \E\sigma(t,\overline{\psi}(0), \overline{\psi}(l)) \to \phi^* \wedge \E\sigma(t'',\overline{\psi}(0), \overline{\psi}(l))  \in \TL{Grz}\oplus AxM.
            \end{align*}
        \end{itemize}
        
        \item $I=t\to\tup{t',0,-1} (\tup{t'',0,0})$. This case is similar to the previous one.

    \end{itemize}
    Thus, we conclude our induction.
\end{proof}

\subsection{Reduction} \label{subsec-3}

Finally, we construct the desired reduction and show the first main theorem, Theorem~\ref{thm:undec-scheme-grzt}. The other main theorem, Theorem~\ref{thm:undec-prop-S4t-=tab} will be shown by modifying this reduction. For each configuration $\tup{t,k,l}$ of $\M$, we define a tense logic $L(t,k,l)$ in a computable way.

\begin{definition}
For each configuration $\tup{t,k,l}$, we define
\setcounter{equation}{0}
\renewcommand{\theequation}{\Roman{equation}}
\begin{align}
    L(t,k,l) \coloneq~ & \TL{Grz} \oplus \axiom{bw}^\partial_5 \oplus \axiom{br}_7 \notag \\
    & \oplus AxM \\
    & \oplus (\phi^* \wedge \E\sigma(s, \overline{\psi}(n), \overline{\psi}(m)) \to \phi^* \wedge \E\sigma(t, \overline{\psi}(k), \overline{\psi}(l))) \to \lnot \phi^* \\ 
    & \oplus \phi^* \to \psi_\tri(\phi_{b_0}, \phi_{b_1}, \phi_{a_0}) \\
    & \oplus \phi^* \to (\psi_\tri(\overline{q}) \to \psi_\tri^+(\overline{q})).
\end{align}
\end{definition}

The intuitive meaning of axioms of $L(t,k,l)$ is as follows.

\begin{itemize}
    \item (I) encodes the instructions in $\M$,
    \item (II) relates the logic $L(t,k,l)$ to the configuration $\tup{t,k,l}$,
    \item (III) finds three points resembling $a_0,b_0,b_1$ in $\gf$,
    \item (IV) forces an infinite descending chain from the three points provided by (III), which would imply Kripke incompleteness.
\end{itemize}

It is clear that the reduction $\tup{t,k,l} \mapsto L(t,k,l)$ is computable. Recall from the statement of Theorem~\ref{thm:undec-scheme-grzt} that, to show the correctness of the reduction, we need to show the following two statements:
\begin{enumerate}
    \item[(*)] if $\tup{t,k,l} \in \Reach$, then $L(t,k,l) = \TL{Grz} \oplus \set{\axiom{bd}_2, \axiom{bw}_2, \axiom{bw}^\partial_5,\axiom{br}_7}$;
    \item[(**)] if $\tup{t,k,l} \not\in \Reach$, then $L(t,k,l)$ is neither Kripke complete nor decidable.
\end{enumerate}

Let us begin with (*), which is easier.

\begin{lemma}\label{lem:tklReach-tabular}
    For any $\tup{t,k,l} \in \Reach$, we have 
    \[L(t,k,l) = \TL{Grz} \oplus \set{\axiom{bd}_2, \axiom{bw}_2, \axiom{bw}^\partial_5,\axiom{br}_7}.\]
\end{lemma}
\begin{proof}
    Let $\tup{t,k,l} \in \Reach$. By Lemma~\ref{lem:S4t-syntax-main} and axiom (II), we have $\lnot \phi^* \in L(t,k,l)$. Since $\phi^* = \lnot \axiom{bd}_2 \lor \lnot \axiom{bw}_2$, this implies $\axiom{bd}_2, \axiom{bw}_2 \in L(t,k,l)$, so the inclusion $\supseteq$ holds. For the other inclusion, it suffices to observe that $\lnot \phi^*$ is provable from $\axiom{bd}_2$ and $\axiom{bw}_2$. All axioms (I), (II), (III), and (IV) are provable from $\lnot \phi^*$. Thus, we conclude the proof.
\end{proof}

The following lemmas show (**), where the general frame $\gf$ is used to witness the Kripke incompleteness and the undecidability.

\begin{lemma} \label{lem:S4t-gf-Ltkl}
    For any $\tup{t,k,l} \not\in \Reach$, we have $\gf \md L(t,k,l)$.
\end{lemma}
\begin{proof}
    It is clear that $\gf \md \axiom{grz} \land \axiom{br}_7$; for the validity of $\axiom{grz}^\partial$, use the fact that admissible sets in $\gf$ are finite or cofinite in $AB = \set{a_i,b_i : i \in \omega}$. To show that $\gf \md \axiom{bw}^\partial_5$, it suffices to observe that there is no $w \in W$ such that $\inver{R}[w]$ contains an anti-chain of more than $5$ points.
    By Lemma~\ref{lem:S4t-gf-AxM}, $\gf \md AxM$. Suppose $\gf, V, w \md \phi^*$ for some valuation $V$ on $\gf$ and $w \in W$. Then $V$ is good. By Lemma~\ref{lem:S4t-formula-location-1}, we obtain $V(\phi_{b_0}) = \set{b_0}$, $V(\phi_{b_1}) = \set{b_{1}}$, and $V(\phi_{a_0}) = \set{a_{0}}$. By Lemma~\ref{lem:S4t-psi-location}, we have $V(\psi_\tri(\phi_{b_0},\phi_{b_1},\phi_{a_0})) = W$, and so $\gf,V,w \md \psi_\tri(\phi_{b_0},\phi_{b_1},\phi_{a_0})$. Thus, $\gf \md \text{(III)}$. Similarly, by Lemma~\ref{lem:S4t-psi+-location}, we have $\gf \md \text{(IV)}$.
    
    It remains to show that $\gf \md \text{(II)}$. Take any $\tup{t,k,l} \not\in \Reach$. Then the point $\tup{t,k,l}_0$ does not exist in $\gf$, while $\tup{s,n,m}_0$ always exists since $\tup{s,n,m} \rightsquigarrow \tup{s,n,m}$. We show that the contrapositive of (II) is valid in $\gf$. Suppose $\gf, V, w \md \phi^*$ for some valuation $V$ on $\gf$ and $w \in W$. Then, $V$ is good. It follows from Lemmas~\ref{lem:S4t-formula-location-1}, \ref{lem:S4t-formula-location-2}, and \ref{lem:S4t-sigma-location} that 
    \[
    \gf, V, \tup{s,n,m}_0 \md \sigma(s, \overline{\psi}(n), \overline{\psi}(m))
    \] 
    and 
    \[
    V(\sigma(t, \overline{\psi}(k), \overline{\psi}(l))) = \ve.
    \]
    Thus, $\gf, V, w \nmd \phi^* \wedge \E\sigma(s, \overline{\psi}(n), \overline{\psi}(m)) \to \phi^* \wedge \E\sigma(t, \overline{\psi}(k), \overline{\psi}(l))$, which entails that $\gf,V,w \md \text{(II)}$. Since $V$ and $w$ were arbitrarily chosen, we conclude that $\gf \md \text{(II)}$.
\end{proof}

Using this lemma, we show that $L(t,k,l)$ is Kripke incomplete (Lemma~\ref{lem:S4t-Ltkl-incomplete}) and undecidable (Lemma~\ref{lem:S4t-Ltkl-undecidable}) for any configuration $\tup{t,k,l} \not\in \Reach$.

\begin{lemma} \label{lem:S4t-Ltkl-incomplete}
    For any $\tup{t,k,l} \not\in \Reach$, the logic $L(t,k,l)$ is Kripke incomplete.
\end{lemma}
\begin{proof}
    Let $\tup{t,k,l}$ be a configuration such that $\tup{t,k,l} \not\in \Reach$. Then, $\lnot \phi^* \notin L(t,k,l)$ since $\gf \md L(t,k,l)$ by Lemma~\ref{lem:S4t-gf-Ltkl} and $\gf \nmd \lnot \phi^*$ by Lemma~\ref{lem:S4t-gf-refutes-phi*}. So, it suffices to show $\neg\phi^*\in\Log(\Fr(L(t,k,l)))$. We show that for any rooted Kripke frame $\ff$ validating $L(t,k,l)$, the refutation of $\neg\phi^*$ would yield descending chains in $\ff$, which contradicts the validity of $\axiom{grz}^\partial$ in $\ff$. See Figure~\ref{fig:RN-construction} for an illustration of the proof.

    Let $\ff = (U, S)$ be a rooted Kripke frame such that $\ff \md L(t,k,l)$. Suppose that $\ff \nmd \neg\phi^*$ for a contradiction. Then $\ff,V,w \md \phi^*$ for some valuation $V$ on $\ff$ and $w \in U$. 
    By (III), we have $\ff,V,w \md \psi_\tri(\phi_{b_0},\phi_{b_1},\phi_{a_0})$. Let $\mm = (\ff,V)$. Then there exist $w_0,u_0 \in U$ such that 
    \begin{align*}
        \mm, w_0 &\md \phi_{a_0} \wedge \phi_{AB} \wedge \B \lnot \phi_{b_0} \land \bb \lnot \phi_{b_0} \land \B \lnot \phi_{b_1} \land \bb \lnot \phi_{b_1}, \\
        \mm, u_0 &\md \phi_{b_0} \wedge \phi_{AB} \wedge \B \lnot \phi_{a_0} \land \bb \lnot \phi_{a_0} \land \B \lnot \phi_{b_1} \land \bd \phi_{b_1}.
    \end{align*}
    By $\mm, u_0 \md \bd\phi_{b_1}$, there exists $u_1 \in \inver{S}[u_0]$ such that $\mm, u_1 \md \phi_{b_1}$. Let $U_1 = \set{w_0,u_0,u_1}$. Then clearly, $S\rsto U_1= \set{(w_0,w_0),(u_0,u_0),(u_1,u_1),(u_1,u_0)}$. Note that $\ff \md \axiom{br}_{7}$. Thus, the modalities $\A$ and $\E$ work as master modalities in $\ff$. Since $\mm,w \md \A(\phi_{b_1} \to \phi_{AB})$, we have $\mm,u_1 \md \phi_{AB}$.
    Since $q, q', q''$ do not occur in any of $\phi^*$, $\phi_{a_0}$, $\phi_{b_0}$, $\phi_{b_1}$, and $\phi_{AB}$ (see Definition~\ref{def:S4t-char-formulas-1}), we may vary the value of $q,q',q''$ while keeping the value of these formulas. For two valuations $V_1$ and $V_2$, we write $V_1 \equiv V_2$ if they agree on propositional variables except for $q,q',q''$.
    
    Let $V' \equiv V$ be a valuation such that $V'(q)=\set{u_0}$, $V'(q')=\set{u_1}$, and $V'(q'')=\set{w_0}$. Let $\mm' = (\ff, V')$. Since $V' \equiv V$ and $S\rsto U_1 = \set{(w_0,w_0),(u_0,u_0),(u_1,u_1),(u_1,u_0)}$, we see that $\mm',w \md \phi^* \wedge \psi_\tri(\overline{q})$. By $\ff \md \text{(IV)}$, we obtain that $\mm',w \md \psi_\tri^+(\overline{q})$, and so $\mm',w \md \psi_\tri(\overline{\tau})$. By the definition of $\psi_\tri$, there exist $w_1$ and $u_2$ such that $\mm',w_1 \md \tau''(\overline{q})$ and $\mm',u_2 \md \tau'(\overline{q})$. Let $U_2 = \set{w_0,u_0,u_1,w_1,u_2}$. It follows from the definition of $\tau'$ and $\tau''$ that $w_1 \neq u_2$ and $w_1,u_2 \not\in U_1$. Moreover, $S\rsto U_2$ is the reflexive-transitive closure of $\set{(u_2,u_1),(u_1,u_0),(w_1,w_0),(u_2,w_0),(w_1,u_0)}$. So, $\phi_{AB}$ is satisfied at $w_1,u_1,u_2$ in $\mm'$.

    Thus, by repeating the argument above (e.g., consider the valuation $V'' \equiv V$ such that $V''(q)=\set{u_1}$, $V''(q')=\set{u_2}$ and $V''(q'')=\set{w_1}$ for the next step), we obtain two infinite descending chains $\set{w_i: i \in \omega}$ and $\set{u_i: i \in \omega}$ in $\ff$. Therefore, by Proposition~\ref{prop:grz}, we have $\ff\nmd \axiom{grz}^\partial$, which contradicts the assumption that $\ff \md L(t,k,l)$. Hence, $\lnot \phi^* \in \Log(\Fr(L(t,k,l)))$, and we conclude that $L(t,k,l)$ is Kripke incomplete.
\end{proof}

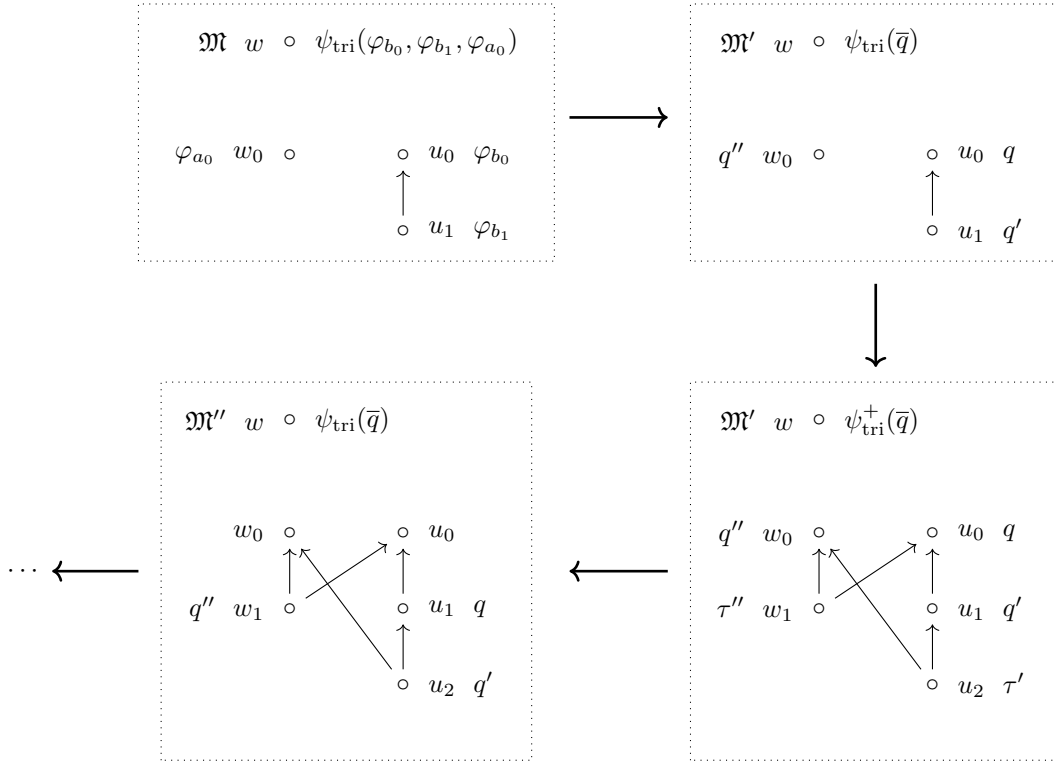
\begin{figure}[ht]
\[
    \begin{tikzpicture}[scale=1]

        \begin{scope}[shift={(-7,0)}]
            \node (w) at (0,4)[label=left:{$\mm~~w$}, label=right:{$\psi_\tri(\phi_{b_0},\phi_{b_1},\phi_{a_0})$}]{$\circ$};
            \node (al0) at (0,2.5)[label=left:$\phi_{a_0}~~w_{0}$]{$\circ$};
            \node (bl0) at (1.5,2.5)[label=right:$u_{0}~~\phi_{b_0}$]{$\circ$};
            \node (bl1) at (1.5,1.5)[label=right:$u_{1}~~\phi_{b_1}$]{$\circ$};
            \draw [->] (bl1) -- (bl0);
            \draw [dotted] (-2,1.1) rectangle (3.5,4.5);
        \end{scope}

        \draw[->, line width=1pt] (-3.3,3) -- (-2,3);

        \begin{scope}[shift={(0,0)}]
            \node (w) at (0,4)[label=left:{$\mm'~~w$}, label=right:{$\psi_\tri(\overline{q})$}]{$\circ$};
            \node (al0) at (0,2.5)[label=left:$q''~~w_{0}$]{$\circ$};
            \node (bl0) at (1.5,2.5)[label=right:$u_{0}~~q$]{$\circ$};
            \node (bl1) at (1.5,1.5)[label=right:$u_{1}~~q'$]{$\circ$};
            \draw [->] (bl1) -- (bl0);
            \draw [dotted] (-1.7,1.1) rectangle (3.2,4.5);
        \end{scope}

        \draw[->, line width=1pt] (.75,.8) -- (.75,-.3);

        \begin{scope}[shift={(0,-5)}]
            \node (w) at (0,4)[label=left:{$\mm'~~w$}, label=right:{$\psi_\tri^+(\overline{q})$}]{$\circ$};

            \node (al0) at (0,2.5)[label=left:$q''~~w_{0}$]{$\circ$};
            \node (al1) at (0,1.5)[label=left:$\tau''~~w_{1}$]{$\circ$};
            \node (bl0) at (1.5,2.5)[label=right:$u_{0}~~q$]{$\circ$};
            \node (bl1) at (1.5,1.5)[label=right:$u_{1}~~q'$]{$\circ$};
            \node (bl2) at (1.50,0.5)[label=right:$u_{2}~~\tau'$]{$\circ$};
            \draw [->] (al1) -- (al0);
            \draw [->] (bl1) -- (bl0);
            \draw [->] (bl2) -- (bl1);
            \draw [->] (al1) -- (bl0);
            \draw [->] (bl2) -- (al0);
            \draw [dotted] (-1.7,-.5) rectangle (3.2,4.5);
        \end{scope}

        \draw[<-, line width=1pt] (-3.3,-3) -- (-2,-3);

        \begin{scope}[shift={(-7,-5)}]
            \node (w) at (0,4)[label=left:{$\mm''~~w$}, label=right:{$\psi_\tri(\overline{q})$}]{$\circ$};

            \node (al0) at (0,2.5)[label=left:$w_{0}$]{$\circ$};
            \node (al1) at (0,1.5)[label=left:$q''~~w_{1}$]{$\circ$};
            \node (bl0) at (1.5,2.5)[label=right:$u_{0}$]{$\circ$};
            \node (bl1) at (1.5,1.5)[label=right:$u_{1}~~q$]{$\circ$};
            \node (bl2) at (1.50,0.5)[label=right:$u_{2}~~q'$]{$\circ$};
            \draw [->] (al1) -- (al0);
            \draw [->] (bl1) -- (bl0);
            \draw [->] (bl2) -- (bl1);
            \draw [->] (al1) -- (bl0);
            \draw [->] (bl2) -- (al0);
            \draw [dotted] (-1.7,-.5) rectangle (3.2,4.5);
        \end{scope}

        \node (x) at (-10.5,-3){$\cdots$};
        \draw[<-, line width=1pt] (x) -- (-9,-3);
    \end{tikzpicture}
\]
\caption{Proof sketch of Lemma~\ref{lem:S4t-Ltkl-incomplete}}
\label{fig:RN-construction}
\end{figure}

\begin{lemma}\label{lem:S4t-Ltkl-undecidable}
    For any $\tup{t,k,l} \not\in \Reach$, the logic $L(t,k,l)$ is undecidable.
\end{lemma}
\begin{proof}
    Take any $\tup{t,k,l} \not\in \Reach$. Then $\gf \md L(t,k,l)$ by Lemma~\ref{lem:S4t-gf-Ltkl}. We reduce the reachability problem $\Reach$ to the decision problem of $L(t,k,l)$. Let $\tup{t',k',l'}$ be an arbitrary configuration. If $\tup{t',k',l'} \in \Reach$, then by Lemma~\ref{lem:S4t-syntax-main} we see that
    \begin{align*}
        &\phi^* \wedge \E\sigma(s, \overline{\psi}(n), \overline{\psi}(m)) \to  \phi^* \wedge \E\sigma(t', \overline{\psi}(k'), \overline{\psi}(l')) \in L(t,k,l).
    \end{align*}
    Suppose $\tup{t',k',l'} \not\in \Reach$. By Lemma~\ref{lem:S4t-gf-refutes-phi*}, there is a good valuation $V$ on $\gf$. Then, similar to the proof of Lemma~\ref{lem:S4t-gf-Ltkl}, it follows from Lemmas~\ref{lem:S4t-formula-location-1}, \ref{lem:S4t-formula-location-2}, and \ref{lem:S4t-sigma-location} that 
    \[\gf, V, \tup{s,n,m}_0 \md \sigma(s, \overline{\psi}(n), \overline{\psi}(m))\] 
    and 
    \[V(\sigma(t', \overline{\psi}(k'), \overline{\psi}(l'))) = \ve.\]
    Since $\gf \md L(t,k,l)$, we have that
    \[\phi^* \wedge \E\sigma(s, \overline{\psi}(n), \overline{\psi}(m)) \to \phi^* \wedge \E\sigma(t', \overline{\psi}(k'), \overline{\psi}(l')) \notin L(t,k,l).
    \]    
    Therefore, since $\Reach$ is undecidable, $L(t,k,l)$ is also undecidable.
\end{proof}

Now we are ready to conclude our proof of the main theorems. We repeat the statements for the convenience of the reader.

\begin{maintheorem}
    Let $P$ be a property satisfying the following two conditions. Then, $P$ is undecidable in $\NExt \TL{Grz}$.
    \begin{enumerate}
        \item The logic $\TL{Grz} \oplus \set{\axiom{bd}_2, \axiom{bw}_2, \axiom{bw}^\partial_5,\axiom{br}_7}$ has $P$,
        \item For any logic $L \in \NExt \TL{Grz}$, if $L$ has $P$, then $L$ is Kripke complete or decidable.
    \end{enumerate}
\end{maintheorem}

\begin{proof}
    The reduction $\tup{t,k,l} \mapsto L(t,k,l)$ is computable. So, it suffices to show the correctness of the reduction. If $\tup{t,k,l} \in \Reach$, then $L(t,k,l)$ is exactly the logic $\TL{Grz} \oplus \set{\axiom{bd}_2, \axiom{bw}_2, \axiom{bw}^\partial_5,\axiom{br}_7}$ by Lemma~\ref{lem:tklReach-tabular}, which has $P$ by the assumption (1). If $\tup{t,k,l} \not\in \Reach$, then $L(t,k,l)$ is neither Kripke complete nor decidable by Lemmas~\ref{lem:S4t-Ltkl-incomplete} and~\ref{lem:S4t-Ltkl-undecidable}. By the assumption (2), $L(t,k,l)$ does not have $P$. Thus, we obtain a reduction from the set $\Reach$, which is undecidable, to the set $\set{\phi\in\FormT:\TL{Grz}\oplus\phi \text{ has } P}$, which is therefore also undecidable.
\end{proof}

To show the second main theorem, we only need to slightly modify the reduction $\tup{t,k,l} \mapsto L(t,k,l)$ for each $d \in \omega$.

\begin{maintheoremtab}
    There are infinitely many tabular tense logics $L \in \NExt \TL{Grz}$ such that coincidence with $L$ is undecidable in $\NExt \TL{Grz}$.
\end{maintheoremtab}

\begin{proof}
    For each $d \in \omega$, let $L_d = \TL{Grz} \oplus \set{\axiom{bd}_2, \axiom{bw}_2, \axiom{bw}^\partial_5,\axiom{br}_{d+7}}$. Every $L_d$ is tabular by Theorem~\ref{thm:grz+bp}. Moreover, these logics are pairwise distinct, since we can always separate them with some finite garlands (see \cite[Chapter 4]{Chen2026a}). We show that coincidence with $L_d$ is undecidable. For each configuration $\tup{t,k,l}$, we define the logic $L_d(t,k,l)$ by first replacing the axiom $\axiom{br}_7$ in $L(t,k,l)$ with $\axiom{br}_{d+7}$ and then redefining $\E$ and $\A$ as $\is{d+7}$ and $\all{d+7}$, respectively. That is,
    \begin{align*}
        L_d(t,k,l) \coloneq~ & \TL{Grz} \oplus \axiom{bw}^\partial_5 \oplus \axiom{br}_{d+7} \notag \\
        & \oplus AxM \\
        & \oplus (\phi^* \wedge \E\sigma(s, \overline{\psi}(n), \overline{\psi}(m)) \to \phi^* \wedge \E\sigma(t, \overline{\psi}(k), \overline{\psi}(l))) \to \lnot \phi^* \\ 
        & \oplus \phi^* \to \psi_\tri(\phi_{b_0}, \phi_{b_1}, \phi_{a_0}) \\
        & \oplus \phi^* \to (\psi_\tri(\overline{q}) \to \psi_\tri^+(\overline{q})).
    \end{align*}
    By repeating the previous arguments, we see that 
    \begin{enumerate}
        \item if $\tup{t,k,l} \in \Reach$, then $L_d(t,k,l) = L_d$,
        \item if $\tup{t,k,l} \not\in \Reach$, then $L_d(t,k,l)$ is Kripke incomplete, which implies that $L_d(t,k,l) \neq L_d$.
    \end{enumerate}
    Thus, $\tup{t,k,l} \mapsto L_d(t,k,l)$ is a reduction from the set $\Reach$ to the set $\set{\phi\in\FormT:\TL{Grz}\oplus\phi = L_d}$. Hence, coincidence with $L_d$ is undecidable for any $d \in \omega$. 
\end{proof}

\section*{Acknowledgments}
We thank Nick Bezhanishvili for his helpful comments on the paper. 
The second author was supported by the Student Exchange Support Program (Graduate Scholarship for Degree Seeking Students) of the Japan Student Services Organization.

\bibliographystyle{plain}
\bibliography{References}

\end{document}